\documentclass{aims} % Use the aims.cls file to compile your paper
\usepackage{amsmath,amssymb,amsthm}
\usepackage{amsmath}
\usepackage{paralist}
\newcommand{\norm}[1]{\left\|#1\right\|}
\newcommand{\abs}[1]{\left\lvert#1\right\rvert}
\usepackage[misc]{ifsym}
\usepackage{epsfig} %For pictures: screened artwork should be set up with an 85 or 100 line screen
\usepackage{epstopdf} %This is to transfer .eps figure to .pdf figure; please compile your article using PDFLaTex or PDFTeXify.
\usepackage[colorlinks=true]{hyperref}
\hypersetup{urlcolor=blue, citecolor=red}
\allowdisplaybreaks

\def\currentvolume{X}
 \def\currentissue{X}
  \def\currentyear{200X}
   \def\currentmonth{XX}
    \def\ppages{X-XX}
     \def\DOI{10.3934/xx.xxxxxxx}

\newtheorem{theorem}{Theorem}[section]
\newtheorem{corollary}[theorem]{Corollary}

\newtheorem{lemma}[theorem]{Lemma}
\newtheorem{proposition}[theorem]{Proposition}

\theoremstyle{definition}
\newtheorem{definition}[theorem]{Definition}
\newtheorem{remark}[theorem]{Remark}

\title[$L^p$- null controllability]
{$L^p$- Partially null controllability of abstract fractional differential inclusion with nonlocal condition} % Only the first word and proper nouns should be capitalized.

\author[Bholanath Kumbhakar and Dwijendra Narain Pandey]{}


\thanks{The first author is supported by Council of Scientific and Industrial Research, New Delhi, Government of India (File No. 09/143(0954)/2019-EMR-I)}

\thanks{$^*$Corresponding author: Dwijendra Narain Pandey}

\begin{document}
\maketitle

% Enter the first author's name and email address; email addresses are required for each author.
% Use footnote notations to indicate address and affiliations with commas between numbers if more than one address applies; see below for examples.
\centerline{\scshape
 Bholanath Kumbhakar$^{{\href{bkumbhakar@mt.iitr.ac.in}{\textrm{\Letter}}}1}$,
  Deeksha$^{{\href{deeksha_s@ma.iitr.ac.in}{\textrm{\Letter}}}1}$
 and Dwijendra Narain Pandey$^{{\href{dwij@ma.iitr.ac.in}{\textrm{\Letter}}}*1}$}

\medskip

{\footnotesize
% Enter the full affiliation and country name:
% Do not insert commas or periods at the end of lines.
 \centerline{$^1$Indian Institute of Technology Roorkee, India}
} % Do not forget to end {\footnotesize with the sign }

\medskip

{\footnotesize
 % Enter the full affiliation and country name:
 \centerline{$^2$Indian Institute of Technology Roorkee, India}
}

\bigskip

% The name of the handling editor will be entered by AIMS production staff.
% "Communicated by Handling Editor" is not needed for special issue.
 \centerline{(Communicated by Handling Editor)}

%%%%%%%%%%%%%%%%%%%%%%%%%%%%%%%%%%%%%%%%%%%%%%%%%%%%%%%
%             5. ABSTRACT
%%%%%%%%%%%%%%%%%%%%%%%%%%%%%%%%%%%%%%%%%%%%%%%%%%%%%%%

\begin{abstract}
	 In this work, we investigate the $L^p$- partial null controllability of the abstract semilinear fractional-order differential inclusion with nonlocal conditions. The set of admissible controls is characterized by $u\in L^p(I,U)$, $1<p<\infty$, $I=[0,\nu]$, where $U$ is a uniformly convex Banach space. Assuming partial null controllability for the fractional-order linear system with a source term, we employ an approximate solvability method to simplify the problem to reduce it to finite-dimensional subspaces. Consequently, the solutions of the original problem are obtained as limiting functions within these subspaces. The paper tackles a challenge stemming from the assumption that $U$ is a uniformly convex Banach space, which introduces convexity issues in constructing the required control.
     These complications do not occur if $U$ is a separable Hilbert space.
     This study introduces a novel approach by resolving the convexity issue, thereby enabling
      $L^p(I, U)$ partially null controllability of the  semilinear fractional-order differential control system, with $U$ being a uniformly convex Banach space.
\end{abstract}

%%%%%%%%%%%%%%%%%%%%%%%%%%%%%%%%%%%%%%%%%%%%%%%%%%%%%%
%                   6. BODY
%%%%%%%%%%%%%%%%%%%%%%%%%%%%%%%%%%%%%%%%%%%%%%%%%%%%%%

% Only the first word and proper nouns of section titles should be capitalized.
% The title of section 1:
\section{Introduction}

Fractional calculus provides a generalized framework where integration and differentiation are defined for non-integer orders. Fractional differential equations are widely encountered in numerous fields, including diffusion processes, electrical engineering, viscoelastic behavior, control of dynamic systems, quantum physics, biological studies, electromagnetic theory, signal analysis, financial modeling, economic theory, and many others (see \cite{hilfer2000applications}, \cite{kilbas2006theory}). On the other hand, fractional differential inclusions find extensive applications in mathematical modeling, especially in economics and optimal control. This has led to extensive research by various authors on the existence of solutions for different types of problems (see
\cite{wang2011existence}, \cite{yan2011nonlocal}). A comprehensive review of the theory of differential inclusion may be found in \cite{deimling2011multivalued} and \cite{kamenskii2011condensing}.

Here, we address the problem of controllability, with a specific focus on null controllability, aiming to bring the solution to the trivial null state in finite time.
The presence of fractional derivatives in time, as we shall observe, leads us to reconsider the controllability analysis from  \cite{matignon1996some}, to include memory effects caused by the tail of the time fractional derivative. Thus, for null controllability to equilibrium, both the state and the memory term must be taken into account. The work in \cite{lu2016lack}, considers the problem of full controllability (refer to definition \ref{fully null controllable}) and demonstrate that, because of memory effects, controllability is unattainable in the fractional setting, even for ODEs.
Even in finite-dimensional systems with full-dimensional control, this negative result still holds. Therefore, these negative results remain valid for fractional-in-time PDEs, irrespective of whether they exhibit parabolic or hyperbolic behavior. Unlike the conclusions drawn in classical ODE and PDE control literature, this negative result presents an entirely opposing behavior.

Based on the above discussion, the present manuscript focuses on investigating the  partial null controllability (see Definition \ref{partially null controlllable}) of abstract fractional-order semilinear differential inclusions with specified nonlocal conditions.
\begin{equation}\label{1.1}
	\begin{cases}
		^{C}D^{\alpha}_{t}q(t)\in \mathbb{A}q(t)+F(t,q(t))+\mathbb{B}u(t), t\in I;\\
		q(0)\in x_0+g(q),
	\end{cases}
\end{equation}  
where $x_0\in X$, $^{C}D^{\alpha}_{t}$ denotes the Caputo fractional derivative of order $\alpha$ with order $\frac{1}{p}<\alpha<1$, $\mathbb{A}$ is the infinitesimal generator of the $C_{0}$ - semigroup $\{T(t)\}_{t\ge 0}$,  consisting of bounded linear operators on a separable reflexive Banach space $X$ with a monotone Schauder basis, $F: I\times X\multimap X$ is a multivalued map, $\mathbb{B}$ is a bounded linear mapping from a uniformly convex Banach space $U$ into $X$, $u\in L^p(I, U) (1<p<\infty)$ and $g: C(I, X)\multimap X$ is a multivalued nonlocal condition.

In order to define the mild solution to the problem \eqref{1.1} we define the measurable selection multimap $S_F:C(I,X)\multimap L^2(I,X)$ as follows: 
\begin{equation}\label{selection}
	S_F(q)=\{ f\in L^2(I;X)~\text{ with}~ f(t)\in F(t,q(t)) ~\text{a.a.}~ t\in I \}.
\end{equation} 
 It will be shown later that, for certain conditions on the multivalued map $F$, the map $S_F$ is well-defined. Moreover, if $F$ has convex values, then $S_{F}$ will also possesses convex values. We now proceed to define the solution concept for equation \eqref{1.1}.
\begin{definition}
	A continuous function $q:I\to X$ is said to be a mild solution of \eqref{1.1} if 
	\begin{equation*}
		q(t)=\mathcal{S}_{\alpha}(t)(x_0+w)+\int_{0}^{t}(t-s)^{\alpha-1}\mathcal{T}_{\alpha}(t-s)f(s)ds+\int_{0}^{t}(t-s)^{\alpha-1}\mathcal{T}_{\alpha}(t-s)\mathbb{B}u(s)ds, t\in I,
	\end{equation*}
	for some $w\in g(q), f\in S_F(q)$.
\end{definition}
Here,
\begin{equation}\label{Talpha}
        \mathcal{S}_{\alpha}(t)=\int^{\infty}_{0}\xi_{\alpha}(\tau)T(t^{\alpha}\tau)d\tau,~~  \mathcal{T}_{\alpha}(t)=\alpha\int^{\infty}_{0}\tau \xi_{\alpha}(\tau)T(t^{\alpha}\tau)d\tau.
    \end{equation}
    Here
    \begin{equation}
         \xi_{\alpha}(\tau)=\frac{1}{\alpha}\tau^{-1-\frac{1}{\alpha}}\overline{w_{\alpha}}(\tau^{-\frac{1}{\alpha}})\geq 0;
    \end{equation}
    where
    \begin{equation}
        \overline{w_{\alpha}}(\tau)=\frac{1}{\pi}\sum^{\infty}_{n=1}(-1)^{n-1}\tau^{-n\alpha-1}\frac{\Gamma(n\alpha+1)}{n!}sin(n\pi\alpha),~ \tau\in(0,\infty),
    \end{equation}
    and  $\xi_{\alpha}$ is a probability density function defined on $(0,\infty)$ , that is 
    \begin{equation}
           \xi_{\alpha}(\tau)\geq0,~ \tau \in (0,\infty)~\text{and}~ \int^{\infty}_{0}\xi_{\alpha}(\tau)d\tau=1.
    \end{equation}

We prove that the inclusion problem \eqref{1.1} is partially null controllable in the following sense.
\begin{definition}\label{partially null controlllable}
We say that the fractional control problem \eqref{1.1} is partially null controllable in $I$ with respect to $L^p$-control if for given $x_0\in X$, there exists $u\in L^p(I,U)$ and a mild solution $q$ of \eqref{1.1} such that $q(0)=x_0$ and $q(\nu)=0$.
\end{definition}
Following \cite{lu2016lack} we introduce the concept of fully null controllability.
\begin{definition}\label{fully null controllable}
We say that the fractional control problem \eqref{1.1} is fully null controllable in $I$ with respect to $L^p$-control if for given $x_0\in X$, there exists $u\in L^p(I,U)$ and a mild solution $q$ of \eqref{1.1} such that $q(0)=x_0$ and $q(t)=0$ for $t\ge \nu$.
\end{definition}
Consider the linear fractional control system
\begin{equation}\label{1.2}
	\begin{cases}
		^{C}D^{\alpha}_{t}q(t)=\mathbb{A}q(t)+\mathbb{B}u(t), t\in I;\\
		q(0)= x_0.
	\end{cases}
\end{equation} 
The authors in \cite{lu2016lack} proved the following.
\begin{theorem}\label{lack of null controllability}
    The linear fractional control system \eqref{1.2} is never fully null controllable for any $\nu>0$ in the sense of Definition \ref{fully null controllable}.
\end{theorem}

    The negative result is unaffected by the functional setting, the generator $\mathbb{A}$ and the control operator $\mathbb{B}$. Moreover, the same conclusion holds for finite-dimensional systems, even when $\mathbb{B}$ is surjective.  Specifically, the first-order scalar fractional differential equation is also not null controllable.

Keeping Theorem \ref{lack of null controllability} in mind, we are interested in partially null controllable for the problem \eqref{1.1}. There are several articles available regarding the partially null controllable in separable Hilbert spaces. In \cite{MR2098245} and \cite{MR3063167}, the authors explored null
controllability in the context of semilinear differential equations within separable
Hilbert spaces.
In \cite{borah2024null}, the author focusing on Null controllability of Fractional differential system with nonlocal intial condition within Hilbert space. In \cite{kumbhakar2025p}, the author explored $L^{p}$ - Null controllability of an abstract differential inclusion with a nonlocal condition within  Banach space.\\
To the best of our knowledge, the study of partially null controllability problems employing $L^{p}(I,U)$ controls within the framework of fractional semilinear evolution inclusions, especially in the presence of general nonlocal conditions, has not yet been thoroughly investigated in the literature. This gap forms the principal motivation for our current research.

Our research expands upon the previously mentioned studies in several significant ways:
\begin{itemize}
	\item[(1)] We consider the control space, denoted as $U$, to be a uniformly convex Banach space rather than the previously assumed separable Hilbert space. Additionally, we consider a broader range of values for $p,$ allowing it to be within the interval $(1, \infty)$, rather than being restricted to $p = 2$.
	\item[(2)] We remove the requirement for the semigroup generated by the operator $\mathbb{A}$ to be compact, which was a prior constraint.
	\item[(3)] It is important to note that a fractional differential inclusion is a more general concept than a fractional differential equation. Therefore, our problem represents an obvious extension of the results presented in references \cite{borah2024null}.
	\item[(4)] We introduce novelty by incorporating nonlocal conditions, which encompass various scenarios such as initial value problems, periodic and anti-periodic problems, more general two-point problems, and numerous other conditions.
\end{itemize}

In summary, our work significantly broadens the scope of research in partially null controllability, addressing various aspects and relaxing prior assumptions to explore uncharted territory in this field.

The paper is neatly structured into six sections:
\begin{itemize}
	\item[(1)] The introduction where we explain the concepts of null controllability.
	\item[(2)] Section 2 revisits preliminary concepts used in obtaining the results.
    \item[(3)] In section 3, we elaborate the existence results for the problem (\ref{1.1}) 
	\item[(4)] In Section 4, we elaborate on the null controllability results.
	\item[(5)] Section 5 presents an application, reinforcing the derived results.
	\item[(6)] Section 6 serves as an appendix.
\end{itemize}
\section{Preliminaries}
We will provide the necessary lemmas and definitions in this section to form the foundation of our study. Since these have been extensively covered in the literature, we will provide the statements along with the necessary references.
Let $(X, \norm{\cdot})$
be a Banach space, with dual space $X^*$  and the duality pairing between $X$ and $X^*$ is denoted by  $\langle \cdot, \cdot\rangle_X$. We consider the Banach spaces 
$\mathcal{L}(X, Y) = \{T: X \rightarrow Y ~|~ \text{T is a bounded linear operator}\}$ and $\mathcal{L}(X) = \{T: X \rightarrow X ~|~ T~ \text{is a bounded linear operator}\}$. Also, throughout the paper integrations will be in Lebesgue-Bochner sense.

 \subsection{Basic definition on fractional calculus}
	We begin this section by outlining definitions, notations, and basic facts concerning fractional calculus, starting with the Riemann-Liouville fractional integral.
    \begin{definition}\cite{kilbas2006theory}
     The Riemann-Liouville fractional integral of order $\alpha$ with $0<\alpha<1$ for a function $f\in L^{1}(I,X)$ is defined as 
     \begin{equation}
         I^{\alpha}f(t)=\frac{1}{\Gamma(\alpha)}\int^{t}_{0}\frac{f(s)}{(t-s)^{1-\alpha}}ds ,~t\in I,
     \end{equation}
     where $\Gamma(\cdot)$ denotes the gamma function.
    \end{definition}
    \begin{definition}\cite{kilbas2006theory}
     The Caputo derivative of order $\alpha , 0 <\alpha<1$ for a function $f\in C^{1}(I,X)$
     is given as 
     \begin{equation}
         ^CD^{\alpha}_{t}f(t)=\frac{1}{\Gamma(1-\alpha)}\int^{t}_{0}\frac{f^{\prime}(s)}{(t-s)^{\alpha}}ds , t\in I.
     \end{equation}
     \end{definition}
     Now consider the linear system
     \begin{equation}\label{frcational linear system}
         \begin{cases}
             	^{C}D^{\alpha}_{t}q(t)= \mathbb{A}q(t)+f(t), t\in I;\\
		q(0)=x_0,
         \end{cases}
     \end{equation}
where the operators $\mathbb{A}$ is as in Problem \eqref{1.1} and $f\in L^2(I,X)$. Then the mild solution of \eqref{frcational linear system} is given by
\begin{equation}
    q(t)=S_{\alpha}(t)x_0+\int_{0}^{t}\mathcal{T}_{\alpha}(t-s)f(s)ds, ~t\in I.
\end{equation}
We consider the following useful Lemma.
 \begin{lemma}\label{Properties of Salpha Talpha}
    The operators $\mathcal{S}_{\alpha}(t)$ and $\mathcal{T}_{\alpha}(t)$ have the following properties:\\
    \begin{itemize}
        \item[(i)] The operators $\mathcal{S}_{\alpha}(t)$ and $\mathcal{T}_{\alpha}(t)$ are linear and bounded . Moreover 
    \begin{equation}
         \norm{\mathcal{S}_{\alpha}(t)}_{\mathcal{L}(X)}\leq M   
    \end{equation}
     and
    \begin{equation}
         \norm{\mathcal{T}_{\alpha}(t)}_{\mathcal{L}(X)}\leq \frac{M}{\Gamma(\alpha)}   
    \end{equation}
    Here, $M=\sup_{t\in I}\norm{T(t)}_{\mathcal{L}(X)}$.
    \item[(ii)] The operators $\mathcal{S}_{\alpha}(t)$ and $\mathcal{T}_{\alpha}(t)$ are strongly continuous for $t\geq 0$.
    \end{itemize}
    \end{lemma}
\subsection{Multivalued Maps}
We briefly discuss some basic results on multivalued maps, and refer the reader to books \cite{deimling2011multivalued} and  \cite{kamenskii2011condensing}
for more in-depth coverage.
\begin{definition}
	A multivalued map $\mathcal{F}:X\multimap Y$ is said to be convex  valued if the image of multivalued map $F$ is convex.
\end{definition}
\begin{definition}
   A multivalued map $\mathcal{F}:X\multimap Y$ is said to be bounded if it maps bounded sets in $X$ into bounded sets in $Y$. 
\end{definition}
The definitions provided below serve to explain the continuity and measurability of multivalued maps.

\begin{definition}\label{Def22}
A multivalued map $\mathcal{F}: X\multimap Y$ is a   upper semicontinuous at the point $x\in \text{Dom}(\mathcal{F})$ if and only if for any neighbourhood $\mathcal{U}$ of $\mathcal{F}(x)$, there exists $\eta >0$  such that $\forall$ $x^{\prime}\in B_{X}(x,\eta)$, $\mathcal{F}(x^{\prime})\subset \mathcal{U}$.
\end{definition}

\begin{remark}\label{rem1}
    when $\mathcal{F}(x)$ is compact, $\mathcal{F}$ is upper semicontinuous at $x$ if and only if $\forall ~ \epsilon,$ there exists $\eta >0$ such that $\forall ~ x^{\prime}\in B_{X}(x,\eta),$ $\mathcal{F}(x^{\prime})\subset B_{Y}(\mathcal{F}(x),\epsilon)$ 
\end{remark}

A corresponding continuous function can be established on weakly compact subsets.
\begin{definition}\label{dww}\cite{kumbhakar2025p}
	Let $X$ be a normed space. A multivalued map $\mathcal{F}: X\multimap X$ is said to be weakly weakly upper semicontinuous if for every weakly closed set $V\subset X$, the set
	\begin{equation}
		\{x\in X: \mathcal{F}(x)\cap V\neq \phi\},
	\end{equation}
	is sequentially weakly closed.
\end{definition}
The characterization below has been formulated in \cite{kumbhakar2025p}.
\begin{lemma}
	A multivalued map $\mathcal{F}: X\multimap X$ is weakly weakly upper semicontinuous at $x$ if for every sequence $\{x_n\}_{n\in \mathbb{N}}$ in $X$ that converges weakly to $x$, and every $y_n\in \mathcal{F}(x_n)$ for $n\in \mathbb{N}$ there exists a subsequence $\{y_{n_k}\}_{k\in \mathbb{N}}$ weakly converging to $y\in \mathcal{F}(x)$.
\end{lemma}
We offer an equivalent definition of weakly upper semicontinuity based on weakly open sets (refer to \cite{kumbhakar2025p})
\begin{lemma}
	Let $X$ be a normed space. If a multivalued map $\mathcal{F}: X\multimap X$ is weakly weakly upper semicontinuous, then for every $x_0\in X$ and for every weakly open set $V\subset X$ containing $\mathcal{F}(x_0)$ there exists an weakly open set $O_X$ containing $x_0$ such that
	\begin{equation}
		\mathcal{F}(x)\subset V ~\forall x\in O_X.
	\end{equation}
\end{lemma}
% \begin{proof}
% 	Let $\mathcal{F}: X\multimap X$ be weakly weakly upper semicontinuous. Fix $x_0\in X$ and let $V$ be weakly open set in $X$ containing $\mathcal{F}(x_0)$. Then $V^c=\{x\in X: x\notin V\}$ is weakly closed set in $X$ and $\mathcal{F}(x_0)\cap V^c=\phi$. Define $C_X=\{x\in X: \mathcal{F}(x)\cap V^c\neq \phi\}$. As the multimap $\mathcal{F}$ is weakly weakly upper semicontinuous, by Definition \eqref{dww} the set $C_X$ is weakly closed in $X$ and hence the set $O_X=C_X^c$ is weakly open in $X$. Further, $x_0\in O_X$ and $\mathcal{F}(x)\subset V$ for $x\in O_X$. Hence, the proof is completed.   
% \end{proof}
The definition of a weakly closed graph associated with a multivalued map is introduced at this stage.
\begin{definition}\cite{kumbhakar2025p}
	Let $\mathcal{F}: X\multimap X$ be a multivalued map. We say that $\mathcal{F}$ has weakly closed graph if for every sequence $\{x_n\}_{n\in \mathbb{N}}, \{y_n\}_{n\in \mathbb{N}}$ in $X$ with $x_n\rightharpoonup x$, $y_n\rightharpoonup y$ in $X$ and $y_n\in \mathcal{F}(x_n)$ for all $n\in \mathbb{N}$, then $y\in \mathcal{F}(x)$.
\end{definition}
Now we can continue with the definition of multivalued measurable functions. Let $X$ be  a separable reflexive Banach space , and let $(\Omega, \Sigma)$ be a measurable space.\\

\begin{definition}
	A multivalued map $\mathcal{F}:\Omega\multimap X$ is said to be measurable if, for each $x\in \Omega$, the set $\mathcal{F}(x)$ is nonempty, closed and for every $z\in X$, the $\mathbb{R}^+$- valued function
	\begin{equation*}
		x\to d(z,\mathcal{F}(x))=\inf\{d(z,w):w\in \mathcal{F}(x)\}
	\end{equation*} 
	is measurable.
	This is equivalent to saying that for every $O_X\subset X$ open, the set 
	\begin{equation*}
		\{w\in \Omega: \mathcal{F}(w)\cap O_X\neq \phi\}\in \Sigma,
	\end{equation*}
	or that there exists a sequence $\{f_n\}_{n\in \mathbb{N}}$ of measurable functions $f_n:\Omega\to X$ such that 
	\begin{equation*}
		\mathcal{F}(w)=\overline{\{f_n(w)\}_{n\ge 1}}, \forall w\in \Omega.
	\end{equation*}
\end{definition}
A Borel $\sigma$- algebra of subsets of $X$ is denoted by $\mathcal{B}(X)$.
\begin{definition}
	We say that the multivalued mapping $\mathcal{F}: \Omega\times X\multimap X$ is $\Sigma\otimes \mathcal{B}(X)$ measurable if
	\begin{equation*}
		\mathcal{F}^{-1}(C_X)=\{(w,x)\in \Omega\times X: \mathcal{F} (w,x)\cap C_X\neq \phi\}\in \Sigma\otimes \mathcal{B}(X),
	\end{equation*}
	for any closed set $C_X\subset X$.\\
\end{definition} 
\begin{proposition}(\cite{MR2024162},proposition 4.1.16)\label{prop2.10}
    If $\mathcal{F}:X\multimap Y$ is closed and locally compact (that is, for every $x\in X$, we can find $U\in \mathcal{N}(x)$ such that $\overline{\mathcal{F}(U)}$ is compact), then $\mathcal{F}$ is upper semicontinuous.
\end{proposition}
Readers seeking further knowledge on multivalued maps may consult the books \cite{papageorgiou1997handbook, MR755330} and the sources referenced within.\\

%	For later use, we state the following lemma and Theorem.
%	\begin{lemma}\label{Lem2.8}\cite{hiai1977integrals}
	%		Let $\Gamma:\Omega\multimap X$ and $\phi:\Omega\times X\to \mathbb{R}$ be $\Sigma\times \mathcal{B}(X)$ measurable. Assume either
	%		\begin{itemize}
		%			\item[(i)] $\phi(w,x)$ is upper semicontinuous in $x$ for every $w\in \Omega$; or
		%			\item[(ii)] $(\Omega, \Sigma,\mu)$ is complete and $\phi(w,x)$ is lower semicontinuous in $x$ for every $w\in \Omega$. 
		%		\end{itemize}
	%		Then the function $w\mapsto \inf\{\phi(w,x):x\in \Gamma(w)\}$ is measuarble.
	%	\end{lemma}
%	\begin{theorem}\label{Thm2.9}\cite{hiai1977integrals}
	%		Let $\Gamma: \Omega\multimap X$ be a closed valued multimap and $1\le p\le \infty$. Let $\phi:\Omega\times X\to \mathbb{R}$ be an $\Sigma\times \mathcal{B}(X)$ measurable function. Define $S_{\Gamma}=\{f\in L^p(\Omega,X): f(\omega)\in \Gamma(\omega)~\text{a.a.}~\omega\in \Omega\}$. If the assumption (i) or (ii) in Lemma \ref{Lem2.8} is supposed and the functional 
	%		\begin{equation*}
		%			I(f)=\int_{\Omega}\phi(w,f(w))d\mu,
		%		\end{equation*}
	%		is defined for all $f\in S_{\Gamma}$ satisfying $I(f_0)<\infty$ for some $f_0\in S_{\Gamma}$, then
	%		\begin{equation}
		%			\inf_{f\in S_{\Gamma}}I(f)=\int_{\Omega}\inf_{x\in \Gamma(w)}\phi(w,x)d\mu.
		%		\end{equation} 
	%	\end{theorem}

\subsection{Compactness in the space $L^p$.}
This section recalls some compactness results in the space $L^p(I, X)$, $1<p<\infty$. We start with the well-known Banach-Alouglu theorem.
\begin{theorem}[Banach-Alaoglu theorem]
	Let $X$ be a Banach space, then the closed unit ball in $X^*$ is weak$^*$-compact.
\end{theorem}
The following Theorem characterizes reflexive Banach spaces. 
\begin{theorem}\label{DP}
	Let $X$ be a Banach space. Then $X$ is reflexive if and only if every bounded sequence has a weakly convergent subsequence.
\end{theorem}
Provided $X$ is reflexive, the space $L^p(I, X)$ is also reflexive for $1<p<\infty$, which gives rise to the next corollary. 
 
\begin{corollary}\label{BA}
	Let $X$ be reflexive. Then, every bounded sequence in the space $L^p(I, X)$ admits a weakly convergent subsequence.
\end{corollary}
\begin{theorem}\label{2.13}
    In a completely metrizable locally convex space, the closed convex hull of a compact
set is compact.
\end{theorem}
  The idea of uniform integrability is introduced as a tool for deriving compactness results in the space $L^1(I,X)$.
\begin{definition}
A sequence $\{f_{n}\}_{n\in \mathbb{N}}\subset L^1(I,X)$ is called uniformly integrable if for every $\epsilon >0$ there is a $\delta(\epsilon)>0$ such that $\int _{E}\norm{f_{n}(s)}ds$ for every measurable subset $E\subset I$ whose Lebesgue $E$
measure is less than or equal to $\delta(\epsilon)$, and uniformly with respect to $n\in \mathbb{N}$.
\end{definition}
\begin{definition}\label{inb}
We say that a sequence $\{f_{n}\}_{n\in \mathbb{N}}$ is integrably bounded if there exists
a $\zeta \in L^1(I,\mathbb{R}^+)$ such that $\norm{f_{n}(t)}\leq \zeta(t)$, for every $n\in \mathbb{N}$ and a.a. $t\in I$.
\end{definition}
Dunford-Pettis Theorem provides the weak compactness result for the space $L^1([0,\nu],X)$.
\begin{theorem}\cite{MR3524637}\label{Dunford Pettis Theorem}
(Dunford Pettis Theorem). Let $X$ be a reflexive Banach space and $\{f_{n}\}_{n\in \mathbb{N}}$ be bounded and uniformly integrable. Then $\{f_{n}\}_{n\in \mathbb{N}}$ is relatively weakly compact in $L^1(I,X)$.
\end{theorem}
\begin{remark}\cite{MR3320658}
In definition \eqref{inb}, if $\{f_{n}\}^{\infty}_{n-1}\subset L^{p}(I,X) ~ (p > 1)$ , $\zeta \in L^{p}(I,\mathbb{R}^+)$ such that $\norm{f_{n}(t)}\leq \zeta(t)$ then $f$ is called $p$-time integrably
bounded.
\end{remark}

\begin{definition}\cite{MR3320658}\label{semicompact}
The sequence $\{f_{n}\}^{\infty}_{n=1}\subset L^{p}(I,X)$ is $p$ - time semicompact if it is $p$-time integrably bounded and the set $\{f_{n}(t)\}^{\infty}_{n=1}$ is relatively compact for a.e. $t\in I$.
\end{definition}

\subsection{Partially Null Controllability}
In this section, we revisit several manifestations of null controllability from the literature, without claiming originality, and explore their interrelationships.

Suppose the linear fractional-order control system 
\begin{equation}\label{NC1}
	\begin{cases}
		^{C}D^{\alpha}_{t}q(t)=\mathbb{A}q(t)+\mathbb{B}u(t), t\in I;\\
		q(0)=x_0.
	\end{cases}
\end{equation}
Here, $^{C}D^{\alpha}_{t}$ is the Caputo fractional derivative of order $\alpha$ with $\frac{1}{p}<\alpha<1$.  The operator $\mathbb{A}$ is a infinitesimal generator of a $C_{0}$ - semigroup $\{T(t)\}_{t\ge 0}$ on Banach space $X$ and $\mathbb{B}$ represent  a bounded linear operator from $U$ to $X$, where $U$ is a uniformly convex Banach space. The control function $u$ belongs to $L^{p}(I,U)$, with $1<p<\infty$. Thus, the influence of the control function $u$ is confined to the range of the operator $\mathbb{B}$. 
% To define a following operator we refer \cite{borah2024null}. \\

An operator $\mathcal{W}:L^p(I,U)\to X$ defined by
\begin{equation}\label{ControlW}
	\mathcal{W}(u)=\int_{0}^{\nu}(\nu-s)^{\alpha-1}\mathcal{T}_{\alpha}(\nu-s)\mathbb{B}u(s)ds, ~ u\in L^p(I,U).
\end{equation}
\begin{definition}
The linear fractional-order control system \eqref{NC1} is partially null controllable with respect to $L^p(I,U)$ with $1<p<\infty$, if for all initial states $x_0\in X$ there exists a control function $u\in L^p(I,U)$ such that 
	
	\begin{equation}
		q(\nu)=\mathcal{S}_{\alpha}(\nu)x_0+\int_{0}^{\nu}(\nu-s)^{\alpha-1}\mathcal{T}_{\alpha}(\nu-s)\mathbb{B}u(s)ds=0,
	\end{equation}
	or, equivalently,
	\begin{equation}
		Im~(\mathcal{S}_{\alpha}(\nu))\subset Im~(\mathcal{W}).
	\end{equation} 
\end{definition}
The following lemma is of fundamental importance \cite{MR2123706}.
\begin{lemma}\label{Lemma210}
	Let $X_0, X, Y $ be Banach spaces. Let $\mathcal{F}\in \mathcal{L}(X,X_0), \mathcal{G}\in \mathcal{L}(Y,X_0)$. Consider the following statements
	\begin{itemize}
		\item[(1)] $Im(\mathcal{G})\subset Im (\mathcal{F})$.
		\item[(2)] There exists $\rho>0$ such that
		\begin{equation*}
			\rho\norm{\mathcal{G}^*x^*}\le \norm{\mathcal{F}^*x^*}, \forall x^*\in X^*.
		\end{equation*}
	\end{itemize}
    Then (1) implies (2). For a reflexive space $X$, (1) holds if and only if (2) holds.
	
\end{lemma}
% Another fascinating result that has been proved in \cite{MR2123706} is given by the following Proposition.
% \begin{proposition}\label{propo}
% 	Let $p\in [1,\infty]$ be fixed. Let us consider the following conditions:
% 	\begin{itemize}
% 		\item[(i)] The control system \eqref{NC1} is null controllable on $I$ with respect to $L^p(I,U)$.
% 		\item[(ii)] There exists $r>0$ such that for $p=1$ we have
% 		\begin{equation*}
% 			r\norm{T^*(\nu)x^*}\le \sup_{t\in I}\norm{\mathbb{B}^*T^*(t)x^*}, \forall x^*\in X^*,
% 		\end{equation*}
% 		and for $1<p\le \infty$ we have
% 		\begin{equation*}
% 			r\norm{T^*(\nu)x^*}\le \left(\int_{0}^{\nu}\norm{\mathbb{B}^*T^*(t)x^*}^{p^{\prime}}dt\right)^{\frac{1}{p^{\prime}}}, \forall x^*\in X^*, ~\text{where}~\frac{1}{p}+\frac{1}{p^{\prime}}=1.
% 		\end{equation*}
% 	\end{itemize}
% 	Then (i)$\Rightarrow$ (ii). For $p\in (1,\infty)$, in the case where $U$ is reflexive, we obtain that the space $L^p(I,U)$ is also reflexive and we have (ii)$\Rightarrow$ (i).\\
% 	For $p=\infty$, (ii)$\Rightarrow$ (i) holds, for instance, when $X$ and $U$ are both reflexive spaces. 
% \end{proposition}

%We now state and prove the following theorem \footnote{J. P. Dauer and N. I. Mahmudov, Exact null Controllability of semilinear integrodifferential
	%	systems in Hilbert spaces, J. Math. Anal. Appl., 299 (2004), 322-332.} 
Next, we turn our attention to the partially null controllability of the linear fractional-order control system with an additive source term
\begin{equation}\label{NC2}
	\begin{cases}
		^{C}D^{\alpha}_{t}q(t)=\mathbb{A}q(t)+\mathbb{B}u(t)+f(t), ~ t\in I;\\
		q(0)=x_0,
	\end{cases}
\end{equation}
where $f\in L^2(I,X)$.\\

We refer \cite{borah2024null} to define the following operator: $\mathcal{Z}: X\times L^2(I,X)\to X$ defined by
\begin{equation*}
	\mathcal{Z}(x_0,f)=\mathcal{S}_{\alpha}(\nu)x_0+\int_{0}^{\nu}(\nu-s)^{\alpha-1}\mathcal{T}_{\alpha}(\nu-s)f(s)ds, ~ x_0\in X, ~ f\in L^2(I,X).
\end{equation*}
\begin{definition}\label{Defc}
The fractional-order control system \eqref{NC2} is partially - null controllable with respect to $L^p(I,U)$ with $1<p<\infty$, if for all initial states $x_0\in X$ there exists a control function $u\in L^p(I,U)$ such that
	
	\begin{equation}
		q(\nu)=\mathcal{S}_{\alpha}(\nu)x_0+\int_{0}^{\nu}(\nu-s)^{\alpha-1}\mathcal{T}_{\alpha}(\nu-s)(f(s)+\mathbb{B}u(s))ds=0,
	\end{equation}
	or, equivalently,
	\begin{equation}
		Im~(\mathcal{Z})\subset Im~(\mathcal{W}).
	\end{equation} 
\end{definition}
We now outline a key condition for the partially null controllability of the fractional-order linear system \eqref{NC2} which follows directly from \cite[Lemma 3.4]{borah2024null}).

\begin{proposition}\label{Prop}
	The fractional-order system \eqref{NC2} is partially null controllable on $I$ if and only if there exists $\gamma>0$ such that
	\begin{equation*}
		\norm{\mathcal{W}^*x^*}\ge \gamma \norm{\mathcal{Z}^*x^*}, \forall x^*\in X^*.
	\end{equation*}
\end{proposition}
%\begin{proof}
%	Suppose the system \eqref{1} is exactly null controllable on $[0, a]$. Then we have $Im(Z)\subset Im (L)$. Then there exists $\gamma>0$ such that
%	\begin{equation*}
	%		\norm{L^*x^*}\ge \gamma \norm{Z^*x^*}, \forall x^*\in X^*.
	%	\end{equation*}
%	Conversely, let us assume the given condition holds. Since $L^2([0, a], U)$ is reflexive, we have the null controllability of the system.
%\end{proof}
From Definition \ref{Defc} and Proposition \ref{Prop}, we establish an equivalent criterion for the partially null controllability of the system\eqref{NC2}.

\begin{proposition}\label{MProp}
	The partially null controllability of the fractional-order system \eqref{NC2} is equivalent to the existence of a $\gamma>0$ such that
	\begin{equation*}
		\norm{(\nu-\cdot)^{\alpha-1}\mathbb{B}^*\mathcal{T}^*_{\alpha}(\nu-\cdot)x^*}_{L^{p^{\prime}}(I,U^*)}\ge \gamma\left[\norm{\mathcal{S}_{\alpha}^*(\nu)x^*}_{X^*}\right]+
         \end{equation*}
    \begin{equation}
        \gamma\left[\norm{(\nu-\cdot)^{\alpha-1}\mathcal{T}_{\alpha}^*(\nu-\cdot)x^*}_{L^2(I,X^*)}\right] ~ \forall x^*\in X^*, 
    \end{equation}
	where, $\frac{1}{p}+\frac{1}{p^{\prime}}=1.$ \\

\textbf{Proof:} Suppose the linear nonhomogeneous fractional- order control system
\begin{equation}
	\begin{cases}
		^{C}D^{\alpha}_{t}q(t)=\mathbb{A}q(t)+\mathbb{B}u(t)+f(t), ~ t\in I;\\
		q(0)=x_0,
	\end{cases}
\end{equation}
where $f\in L^2(I,X)$.\\
Assume the system is partially-null controllable. According to Proposition \ref{Prop}, it follows that
\begin{equation*}
		\norm{\mathcal{W}^*x^*}\ge \gamma \norm{\mathcal{Z}^*x^*}, \forall x^*\in X^*.
	\end{equation*}
Recall an operator $\mathcal{W}:L^p(I,U)\to X$ defined by
\begin{equation*}
	\mathcal{W}(u)=\int_{0}^{\nu}(\nu-s)^{\alpha-1}\mathcal{T}_{\alpha}(\nu-s)\mathbb{B}u(s)ds, ~ u\in L^p(I,U).
\end{equation*}

Then, adjoint of $\mathcal{W}$ is given by:\\
\begin{align*}
    \langle x^*, \mathcal{W}(u)\rangle =& \langle x^*, \int ^{\nu}_{0}(\nu-s)^{\alpha-1}\mathcal{T}_{\alpha}(\nu-s)\mathbb{B}u(s)ds\rangle_{X^*\times X}\\
    =& \int^{\nu}_{0}(\nu-s)^{\alpha-1}\langle x^*, \mathcal{T}_{\alpha}(\nu-s)\mathbb{B}u(s) \rangle_{X^*\times X}ds\\
    =& \int ^{\nu}_{0}(\nu-s)^{\alpha-1}\langle \mathbb{B}^*\mathcal{T}^*_{\alpha}(\nu-s)x^*, u(s)\rangle_{U^*\times U}ds\\
    =& \int^{\nu}_{0}\langle(\nu-s)^{\alpha-1}\mathbb{B}^*\mathcal{T}^*_{\alpha}(\nu-s)x^*, u(s)\rangle_{U^*\times U}ds\\
    =& \langle (\nu-\cdot)^{\alpha-1}\mathbb{B}^*\mathcal{T}^*_{\alpha}(\nu-\cdot)x^*, u(\cdot)\rangle_{L^{p^{\prime}}(I,U^*)\times L^{p}(I,U)}\\
    =& \langle \mathcal{W}^*x^*, u\rangle_{L^{p^{\prime}}(I,U^*)\times L^{p}(I,U)}
\end{align*}

Therefore, adjoint $\mathcal{W}^*:X^*\rightarrow L^{p^{\prime}}(I,U^*)$ of $\mathcal{W}$ is given by 
\begin{equation}
    \mathcal{W}^*(x^*) = (\nu-\cdot)^{\alpha-1}\mathbb{B}^*\mathcal{T}^*_{\alpha}(\nu-\cdot)x^*
\end{equation}

Recall an operator $\mathcal{Z}: X\times L^2(I,X)\to X$ defined by
\begin{equation*}
	\mathcal{Z}(x_0,f)=\mathcal{S}_{\alpha}(\nu)x_0+\int_{0}^{\nu}(\nu-s)^{\alpha-1}\mathcal{T}_{\alpha}(\nu-s)f(s)ds, ~ x_0\in X, ~ f\in L^2(I,X).
\end{equation*}
\begin{remark}
    Consider the space $X\times L^{2}(I,X)$ , which is a Banach space under the norm 
    \begin{equation}
        \norm{(x,f)}_{X\times L^{2}(I,X)}=\norm{x}_{X}+\norm{f}_{L^{2}(I,X)}, ~ (x,f)\in X\times L^{2}(I,X).
    \end{equation}
    The duality pairing in this space can be expressed as 
    \begin{equation}
        \langle(x,f),(y,g)\rangle=\langle x,y\rangle_{X\times X^*}+\langle f,g\rangle_{L^{2}(I,X)\times L^{2}(I,X^*)},
    \end{equation}
    for all $(x,f)\in L^{2}(I,X), (y,g)\in X^*\times L^{2}(I,X^*)$.
\end{remark}
Then, adjoint of $\mathcal{Z}$ is given by: 
\begin{align*}
    \langle x^*, \mathcal{Z}(x_0,f)\rangle_{X^*\times X} =&\langle x^*, \mathcal{S}_{\alpha}(\nu)x_{0}+\int ^{\nu}_{0}(\nu-s)^{\alpha-1}\mathcal{T}_{\alpha}(\nu-s)f(s)ds\rangle \\
    =& \langle x^*, \mathcal{S}_{\alpha}(\nu)x_{0}\rangle + \langle x^*, \int^{\nu}_{0}(\nu-s)^{\alpha-1}\mathcal{T}_{\alpha}(\nu-s)f(s)ds \rangle\\
    =& \langle\mathcal{S}^*_{\alpha}(\nu)x^*,x_{0}\rangle+\int^{\nu}_{0}(\nu-s)^{\alpha-1}\langle \mathcal{T}^*_{\alpha}(\nu-s)x^*,f(s)\rangle ds\\
    =& \langle \mathcal{S}^*_{\alpha}(\nu)x^*, x_{0}\rangle + \langle (\nu-\cdot)^{\alpha-1}\mathcal{T}^*_{\alpha}(\nu-\cdot)x^*, f\rangle_{L^{2}(I,X^*),L^{2}(I,X)}\\
    =& \langle (\mathcal{S}^*_{\alpha}(\nu)x^*,(\nu-\cdot)^{\alpha-1}\mathcal{T}^*_{\alpha}(\nu-\cdot)x^*), (x_{0},f)\rangle\\
    =&\langle \mathcal{Z}^*x^*, (x_{0},f)\rangle
\end{align*}

Therefore, adjoint operator  $\mathcal{Z}^*:X^* \rightarrow X^*\times L^{2}(I,X^*)$ of $\mathcal{Z}$ is given by :
\begin{equation}
    \mathcal{Z}^*(x^*) = (\mathcal{S}^*_{\alpha}(\nu)x^*,(\nu-\cdot)^{\alpha-1}\mathcal{T}^*_{\alpha}(\nu-\cdot)x^*)
\end{equation}
 Substituting the value of $\mathcal{W}^*$ and $\mathcal{Z}^*$ in proposition \eqref{Prop}, then we obtain\\

 \begin{equation*}
		\norm{(\nu-\cdot)^{\alpha-1}\mathbb{B}^*\mathcal{T}^*_{\alpha}(\nu-\cdot)x^*}_{L^{p^{\prime}}(I,U^*)}\ge \gamma\norm{\mathcal{S}_{\alpha}^*(\nu)x^*}_{X^*}+
         \end{equation*}
    \begin{equation}
        \gamma\norm{(\nu-\cdot)^{\alpha-1}\mathcal{T}_{\alpha}^*(\nu-\cdot)x^*}_{L^2(I,X^*)} ~ \forall x^*\in X^*, 
    \end{equation}
	where, $\frac{1}{p}+\frac{1}{p^{\prime}}=1$ \\
 
\end{proposition}

\textbf{Natural projection map and monotone schauder basis :} A sequence $\{e_n\}_{n\in \mathbb{N}}$ in infinite dimensional normed linear space $X$ is called a Schauder basis if for every $x\in X$ there exist a unique sequence of scalars $\{a_n\}_{n\in \mathbb{N}}$, such that $\displaystyle x=\sum_{n=1}^{\infty}a_ne_n$, where $\{e_n\}_{n\in \mathbb{N}}$ is a linearly independent set in $X$. The presence of a Schauder basis in a Banach space $X$ ensure that $X$ is separable, as the set of finite rational linear combinations of the basis vectors is countable and dense.
\\
Define the natural projections $P_n:X\to X_n$ by
\begin{equation}\label{Proje}
	P_n\left(\sum_{i=1}^{\infty} a_ix_i\right)=\sum_{i=1}^{n}a_ix_i.
\end{equation}

Where, $X$ is a separable reflexive Banach space with a Schauder basis $\{x_1,x_2,\cdot\cdot\cdot, x_n,\cdot \cdot \cdot\}$ and $X_{n}$ is an $n$-dimensional Banach space which is span by $\{x_1,x_2,\cdot\cdot\cdot, x_n\}$. 

\begin{proposition}
    Let $X$ be a Banach space with a Schauder basis $\{x_{n}\}_{n\in\mathbb{N}}$. Then the natural projections $\eqref{Proje}$ are bounded linear operators and $\text{sup}_{n}\norm{P_{n}}< \infty$. 
\end{proposition}

The number $\text{sup}_{n}\norm{P_{n}}$ is called the basis constant. A basis whose basis constant is one, is called a monotone Schauder basis. In other words, a Schauder basis is monotone if, for every choice of scalars $\{a_{n}\}^{\infty}_{n=1}$, the sequence of numbers $\{\norm{\sum^{n}_{i=1}a_{i}x_{i}}\}^{\infty}_{n=1}$ is non-decreasing. Every Schauder basis is monotone with respect to the norm $\norm{x}_{s} = \text{sup}_{n}\norm{P_{n}x}$.\\

The subsequent results address the convergence behavior of natural projection maps.
\begin{proposition}\label{Prop1}\cite{malaguti2021begin}
	Let $\{x_n\}_{n\in \mathbb{N}}$ be a Schauder basis of the Banach space $X$ and $\{P_n\}_{n\in \mathbb{N}}$ be the sequence of natural projections defined by \eqref{Proje}. Then the following hold:
	\begin{enumerate}
		\item there exists $K\ge 1$ such that $$\norm{P_n(x)}\le K\norm{x},$$ for all $n\in \mathbb{N}$ and $x\in X$.
		\item for every $x\in X$, $\norm{P_n(x)-x}\to 0$ as $n\to \infty$.
	\end{enumerate}
\end{proposition}
\begin{proposition}\label{Prop2}\cite{malaguti2021begin}
	Let $X$ be a Banach space with a Schauder basis $\{x_k\}_{k\in \mathbb{N}}$ and $\{P_n\}_{n\in \mathbb{N}}$ be the sequence of natural projections. Then
	\begin{equation}\label{2.2}
		\text{if}~ y_k\rightharpoonup y~\text{then}~ P_n(y_k)\to P_n(y), ~\text{as}~ k\to \infty,~\text{ for every}~ n\in \mathbb{N},
	\end{equation}
	and
	\begin{equation}\label{2.3}
		\text{if}~ y_k\to y~\text{ then}~ P_k(y_k)\to y,~\text{ as}~ k\to \infty.
	\end{equation}
	
	Moreover,
	\begin{equation}\label{2.4}
		\text{if}~X~\text{is reflexive and}~y_k\rightharpoonup y~ \text{then}~ P_k(y_k)\rightharpoonup y~\text{as}~k\to \infty.
	\end{equation}
\end{proposition}
For our analysis, we refer to the fixed point theorem in \cite{pinaud2020controllability}. 

\begin{theorem}\label{fixed}
	Let $K$ be a weakly compact convex subset of a Banach space $X$, and let $\Gamma:K\multimap K$ be a weakly weakly upper semicontinuous multivalued mapping with closed convex values. Then the multimap $\Gamma$ has a fixed point.
\end{theorem}
\subsection{Hypotheses}
In the present article, $w-X$ denotes the Banach space $X$ equipped with the weak topology, and the subsequent hypothese will be considered:

\begin{itemize}
	\item[\textbf{(T)}] A $C_{0}$- semigroup $\{T(t)\}_{t\ge0}$ on the space $X$ is generated by the operator $\mathbb{A}$. Additionaly, there exists a positive constant $M$ such that the semigroup norm is bounded by $M$ for all non-negative $t$ i.e.
     \begin{equation*}
		\norm{T(t)}_{\mathcal{L}(X)}\le M, ~ t\ge 0.
	\end{equation*} 
	\item[\textbf{(U)}] Consider the linear fractional-order system with an additive source term linked to problem \eqref{1.1}, is formulated as
	\begin{equation*}
		\begin{cases}
			^{C}D^{\alpha}_tq(t)=\mathbb{A}q(t)+\mathbb{B}u(t)+f(t), t\in I;\\
			q(0)=x_0,
		\end{cases}
	\end{equation*}
	is partially null controllable in $I$ for any $f\in L^2(I,X)$.
	\item[(\textbf{F1})] The  values of the multivalued nonlinearity $F : I \times X \multimap X$ are nonempty, weakly closed and convex. 
	\item[(\textbf{F2})]  For every $N\in \mathbb{N}$ there exists $\eta_N\in L^\frac{1}{\alpha_1}(I;\mathbb{R}^+)$ with $\alpha_1<\alpha$ such that
	\begin{equation*}
		\norm{F(t,y)}_X\le \eta_N(t), ~\text{a.a.}~t\in I,~\text{and for every}~y\in X, \norm{y}_X\le N,		
	\end{equation*}
	
	with 
    \begin{equation*}
		\liminf_{N\to \infty}\frac{1}{N}\int_{0}^{\nu}(\nu-s)^{\alpha-1}\eta_N(s)ds=0.
	\end{equation*}
	
	%		\item[(\textbf{F3})] The function $F:I\times X\multimap X$ is graph measurable.
	\item[(\textbf{F3})] The function $F(\cdot,y):I\multimap X$ has a measurable selection for every $y\in X$.
	%		\item[(\textbf{F4})] For almost all $t \in I$, the function $F(t,\cdot): X\multimap X$ is lower semicontinuous.
	\item[(\textbf{F4})] For almost all $t \in I$, the multimap $F(t,\cdot): X\multimap X$ is weakly weakly upper semicontinuous. 
	%		\item[(\textbf{F4(l)})] For almost all $t \in I$, the function $F(t,\cdot): X\multimap X$ is  lower semicontinuous.
	%		\item[(F5)] For every bounded subset $\Omega\subset X$, there exists $k_{\Omega}\in L^1(I,\mathbb{R}^+)$ such that 
	%		\begin{equation*}
		%			\beta(F(t,\Omega))\le k(t)\beta(\Omega), ~\text{a.a.}~t\in I.
		%		\end{equation*}
	
	The multimap $g: C(I,X)\multimap X$ be such that
	\item[(\textbf{g1})]$g$ is weakly closed convex valued multimap and $g$ maps bounded set in $C(I,X)$ into bounded sets in $X$ and 
	\begin{equation*}
		\lim\limits_{\norm{q}_{C(I,X)}\to \infty}\frac{\norm{g(q)}}{\norm{q}_{C(I,X)}}=0.
	\end{equation*}
	\item[(\textbf{g2})] The function $g$ has a weakly closed graph in the following sense:\\
	Suppose $\{w_n\}_{n\in \mathbb{N}}\subset X, \{q_n\}_{n\in \mathbb{N}}\subset C(I,X)$ with $w_n\rightharpoonup w$ in $X$ and $q_n\rightharpoonup q$ in $C(I,X)$. If $w_n\in g(q_n),$ then $w\in g(q)$. 
\end{itemize}
\begin{remark}{\label{R2.25}}
    Suppose Hypothesis (\textbf{g1}) hold. Then following [\cite{malaguti2021begin}, proposition 3] we have
    \begin{equation}
        \lim_{N\rightarrow \infty}\frac{g_{N}}{N}=0,~ \text{where}~ g_{N} = \text{sup}\{\norm{g(q)}:\norm{q}_{C(I,X)}\leq N\} 
    \end{equation}
\end{remark}

\begin{remark}
Demonstrating the partially null controllability of the fractional-order linear control system
	
	\begin{equation}\label{2.20}
		\begin{cases}
			^{C}D^{\alpha}_tq(t)=\mathbb{A}q(t)+\mathbb{B}u(t), t\in I;\\
			q(0)=x_0,
		\end{cases}
	\end{equation}
 does not automatically ensure the partially null controllability of the system presented in equation \eqref{1.1}. This constraint arises because the proper definition of the control for equation \eqref{1.1}, requires the satisfaction of the following condition.    
	 \begin{equation}
		\mathcal{S}_{\alpha}(\nu)x_0+\int_{0}^{\nu}(\nu-s)^{\alpha-1}\mathcal{T}_{\alpha}(\nu-s)f(s)ds\in Im(\mathcal{W}), ~\text{for any }~x_0\in X, f\in L^2(I,X).
	\end{equation}
	This guarantee is achievable only when Assumption (\textbf{U}) is fulfilled.
\end{remark}

\subsection{Measurable selection Map}
Consider the measurable selection map $S_F:C(I,X)\multimap L^2(I,X)$ is defined by
\begin{equation*}
	S_F(q)=\{ f\in L^2(I;X)~| ~ f(t)\in F(t,q(t)) ~\text{for almost every}~ t\in I \}.
\end{equation*}
The upcoming theorems confirm that $S_F(q)$ is nonempty for every $q\in C(I, X)$. The subsequent theorems explore some properties of this set. \\

\begin{theorem}\cite{kumbhakar2025p}\label{Thm2.29}
	Suppose the multivalued map $F:I\times X \multimap X$ satisfies the hypotheses (\textbf{F1})-(\textbf{F4}).
	Then, the multifunction $S_F$ has nonempty and weakly compact convex values.
\end{theorem}

\begin{theorem}\cite{kumbhakar2025p}\label{SFT}
	Let $X$ be a Banach space and let $F:I\times X\multimap X$ be a multivalued function such that (\textbf{F1})-(\textbf{F4}) are satisfied. Let $q_m\in C(I,X)$, $f_m\in L^2(I,X)$ be such that $q_m(t)\rightharpoonup q(t)$ in $X$, $f_m\rightharpoonup f$ in $L^2(I,X)$. If $f_m\in S_F(q_m)$, then $f\in S_F(q)$.
\end{theorem}
\section{Existence Result:}
The following theorem proves the existence results to the problem \eqref{1.1}\\
\begin{theorem}
    Fix $u\in L^{p}(I,U)$. Assume that Hypothesis (\textbf{T}) holds. Also, assume that the
multivalued nonlinearity $F$ satisfies Hypotheses (\textbf{F1-F4}) and the nonlocal multifunction $g$ satisfies hypotheses (\textbf{g1-g2}). Then, the system \eqref{1.1} has atleast one mild solution.
\end{theorem}
\begin{proof}
    Let us define 
    \begin{equation}
        B^{N}_{n} =\{q\in C(I,X_{n}): \norm{q(t)}\leq N, ~ t\in I\}, ~ N>0,
    \end{equation}
    where $X_{n} = \text{span}\{x_1,x_2,...,x_n\}$. Fix $n\in \mathbb{N}$ and $u(\cdot)\in L^{p}(I,U)$. \\
    Define multimap $\Lambda_{n} : C(I,X_{n})\multimap C(I,X_{n})$ as follows: for $q\in C(I,X_{n})$,
    if $y\in \Lambda_{n}(q)$, then 
    \begin{equation}{\label{T33}}
        y(t) = P_n\mathcal{S}_{\alpha}(t)(x_{0}+w)+\int ^{t}_{0}(t-s)^{\alpha-1}P_{n}\mathcal{T}_{\alpha}(t-s)P_{n}f(s)ds + \int ^{t}_{0}(t-s)^{\alpha-1}P_{n}\mathcal{T}_{\alpha}(t-s)Bu(s)ds
    \end{equation}
    $ t\in I,$
    where $w\in g(q)$ and $f\in S_{F}(q)$. Here $P_n$, $n\in \mathbb{N}$ denotes the natural projections defined in \eqref{Proje}. We
complete the proof in several steps.\\

\textbf{\textit{Step-(i)}:} In this step we show that the multimap $\Lambda_{n}$ maps $B^{N_{0}}_{n}$ into $B^{N_{0}}_{n}$ for some $N_{0}>0$. Let $y\in \Lambda_{n}(q)$, $q\in B^N_{n}$. By the definition of the map $\Lambda_{n}$ we obtain $y$ satisfies \eqref{T33}. We now estimate \\
\begin{align*}
    \norm{y(t)}&\leq \norm{P_{n}\mathcal{S}_{\alpha}(t)(x_0+w)}+\int^{t}_{0}(t-s)^{\alpha-1}\norm{P_n\mathcal{T}_{\alpha}(t-s)P_{n}f(s)}ds\\ 
     &+ \int^{t}_{0}(t-s)^{\alpha-1}\norm{P_{n}\mathcal{T}_{\alpha}(t-s)Bu(s)}ds\\
    &\leq M\norm{x_{0}+w}+\frac{M\kappa_1}{\Gamma(\alpha)}\norm{\eta_{N}}_{L^{\frac{1}{\alpha_1}}(I;\mathbb{R}^+)}+\frac{M\kappa_2}{\Gamma(\alpha)}\norm{\mathbb{B}}\norm{u}_{L^{p}(I;U)}\\
    &\leq M\norm{x_{0}}+M\norm{w}+\frac{M\kappa_1}{\Gamma(\alpha)}\norm{\eta_{N}}_{L^{\frac{1}{\alpha_1}}(I;\mathbb{R}^+)}
    +\frac{M\kappa_2}{\Gamma(\alpha)}\norm{\mathbb{B}}\norm{u}_{L^{p}(I;U)}\\
    &\leq M\norm{x_{0}}+M\norm{g(B^{N}_{n})}+\frac{M\kappa_1}{\Gamma(\alpha)}\norm{\eta_{N}}_{L^{\frac{1}{\alpha_1}}(I;\mathbb{R}^+)}
    +\frac{M\kappa_2}{\Gamma(\alpha)}\norm{\mathbb{B}}\norm{u}_{L^{p}(I;U)}
\end{align*}
Here, $\kappa_1= \Big(\frac{\nu^{\frac{\alpha-\alpha_1}{1-\alpha_1}}}{\frac{\alpha-\alpha_1}{1-\alpha_1}}\Big)^{1-\alpha_1}$ and $\kappa_2= \Big(\frac{\nu^{p^{\prime}(\alpha-1)+1}}{p^{\prime}(\alpha-1)+1}\Big)^{\frac{1}{p^{\prime}}}$. We use H$\ddot{\text{o}}$lder's inequality, $\frac{1}{p}+\frac{1}{p^{\prime}}=1$, and $\norm{g(B^{N}_{n})} = \text{sup}\{\norm{w}:w\in g(q), ~ q\in B^{N}_{n}\}$. Therefore, we obtain
\begin{align*}
    \norm{\Lambda_{n}(q)}\leq &M\norm{x_{0}}+M\norm{g(B^{N}_{n})}
+\frac{M}{\Gamma(\alpha)}\kappa_1\norm{\eta_{N}}_{L^{\frac{1}{\alpha_1}}(I;\mathbb{R}^+)} 
+\frac{M}{\Gamma(\alpha)}\norm{\mathbb{B}}\kappa_2\norm{u}_{L^{p}(I;U)}
    \end{align*}
Consequently, 
\begin{equation*}
    \norm{\Lambda_{n}(B^{N}_{n})}\leq M\norm{x_{0}}+M\norm{g(B^{N}_{n})}+\frac{M}{\Gamma(\alpha)}\kappa_1\norm{\eta_{N}}_{L^{\frac{1}{\alpha_1}}(I;\mathbb{R}^+)}  +\frac{M}{\Gamma(\alpha)}\norm{\mathbb{B}}\kappa_2\norm{u}_{L^{p}(I;U)}.
\end{equation*}
In view of Remark \ref{R2.25} and on the basis of Hypothesis (\textbf{F2}) and (\textbf{g1}) we deduce that
\begin{equation}
    \lim_{N\rightarrow \infty}\text{inf} \frac{1}{N}\norm{\eta_{N}}_{L^{\frac{1}{\alpha_1}}(I;U)}=0 ~ \text{and} \lim_{N\rightarrow \infty} \frac{\norm{g(B^{N}_{n})}}{N}=0
\end{equation}
Based on this observation, we can deduce that
\begin{align*}
    0&\leq\liminf_{N\rightarrow \infty}\frac{1}{N}\norm{\Lambda_{n}(B^{N}_{n})}\\
    &\leq\liminf_{N\rightarrow \infty}\frac{M}{N}\Big[\norm{x_{0}}+\norm{g(B^{N}_{n})}
    +\frac{\kappa_1}{\Gamma(\alpha)}\norm{\eta_{N}}_{L^{\frac{1}{\alpha_1}}(I;\mathbb{R}^+)}
    +\frac{\kappa_2}{\Gamma(\alpha)}\norm{\mathbb{B}}\norm{u}_{L^{p}(I;U)}\Big]=0
\end{align*}
This implies
\begin{equation}
    \lim_{N\rightarrow \infty}\text{inf}\frac{1}{N}\norm{\Lambda_{n}(B^{N}_{n})}=0
\end{equation}
It remains to prove that there exists $N_{0}\in \mathbb{N}$ such that $\Lambda_{n}(B^{N_0}_{n})\subset B^{N_{0}}_{n}$ for every $n\in \mathbb{N}$. Suppose, by
contradiction, that for every $N\in \mathbb{N}$ there exists $\overline{n}\in \mathbb{N}$ and $y\in \Lambda_{\overline{n}}(q)$, $q\in B^{N}_{\overline{n}}$ such that $y\notin B^{N}_{\overline{n}}$. By the fact $y\notin B^{N}_{\overline{n}}$ we have 
\begin{equation*}
    N<\norm{\Lambda_{\overline{n}}(q)}\leq M\norm{x_{0}}+M\norm{g(B^{N}_{n})}+\frac{M\kappa_1}{\Gamma(\alpha)}\norm{\eta_{N}}_{L^{\frac{1}{\alpha_1}}(I;\mathbb{R}^+)} +\frac{M\kappa_2}{\Gamma(\alpha)}\norm{\mathbb{B}}\norm{u}_{L^{p}(I;U)},
\end{equation*}
which implies
\begin{equation*}
    1<\frac{1}{N}\norm{\Lambda_{\overline{n}}(q)}\leq \frac{M\norm{x_{0}}+M\norm{g(B^{N}_{n})}+\frac{M}{\Gamma(\alpha)}\kappa_1\norm{\eta_{N}}_{L^{\frac{1}{\alpha_1}}(I;\mathbb{R}^+)}}{N} +\frac{\frac{M}{\Gamma(\alpha)}\norm{\mathbb{B}}\kappa_2\norm{u}_{L^{p}(I;U)}}{N}.
\end{equation*}
We reach a contradiction by passing the limit as $N\rightarrow \infty$. This concludes that the map $\Lambda_{n}$ maps $B^{N_0}_{n}$ into $B^{N_0}_{n}$ for each $n\in \mathbb{N}$.\\
\vspace{0.5cm}

\textbf{\textit{Step-(ii)}:} In this step we show the map $\Lambda_{n}$ maps the set $B^{N_0}_{n}$ into a relatively compact set in $C(I,X_{n})$. Let $y_{m}\in \Lambda_{n}(q_{m})$, $q_{m}\in B^{N_0}_{n}$, then the following equation holds for $t\in I$ : 
\begin{equation*}
    y_{m}(t)= P_{n}\mathcal{S}_{\alpha}(t)(x_{0}+w_{m}) + \int^{t}_{0}(t-s)^{\alpha-1}P_{n}\mathcal{T}_{\alpha}(t-s)P_{n}f_{m}(s)ds 
    \end{equation*}
    \begin{equation}\label{EQE}
         + \int^{t}_{0}(t-s)^{\alpha-1}P_{n}\mathcal{T}_{\alpha}(t-s)Bu(s)ds, ~ t\in I, 
    \end{equation}
    where $w_{m}\in g(q_{m})$ and $f_{m}\in S_{F}(q_{m})$ for all $m\in \mathbb{N}$.
Owing to Hypothesis \textbf{(g1)} we confirm that the sequence $\{w_{m}\}_{m\in \mathbb{N}}$ is bounded in $X$. As $X$ is reflexive, we conclude that $w_{m}\rightharpoonup w$ in $X$ up to a subsequence. Noting that the operator $\mathcal{S}_{\alpha}(t)$ is bounded linear, by virtue of Proposition \ref{Prop2} we obtain 
\begin{equation}\label{42}
    P_{n}\mathcal{S}_{\alpha}(t)(x_{0}+w_{m})\rightarrow P_{n}\mathcal{S}_{\alpha}(t)(x_{0}+w), ~ as ~ m\rightarrow \infty.
\end{equation}
From the fact that $f_{m}\in S_{F}(q_{m})$ and Hypothesis $\textbf{(F2)}$ we conclude that the sequence $\{f_{m}\}_{m\in \mathbb{N}}$ is
integrably bounded in $L^{2}(I,X)$. In accordance with the reflexivity of the Banach space $X$, we conclude that $\{f_{m}\}_{m\in \mathbb{N}}$ is weakly relatively compact in $L^2(I, X)$. Hence, $f_{m}\rightharpoonup f$ in $L^2(I,X)$ up to a subsequence.\\
We define $h_{m}\in L^2(I,X)$ such that $h_{m}(t)=P_{n}f_{m}(t)$ for $t\in I$, and $m\in \mathbb{N}$. For almost every $t\in I$, the set $\{h_{m}(t):m\in \mathbb{N}\}$ is bounded and lies within the finite-dimensional space $X_{n}$. Hence, it
is relatively compact in both $X_{n}$ and $X$. Additionally, $\{h_{m}\}_{m\in\mathbb{N}}$ is integrably bounded as $\{f_{m}\}_{m\in \mathbb{N}}$  is integrably bounded. Therefore, based on the Definition \ref{semicompact}, $\{h_{m}\}$ is a semicompact sequence. Further, $\{h_{m}\}$ converges weakly to $\tilde{f}$ in $L^2(I,X)$, where $\tilde{f}(t)=P_{n}f(t)$, a.a. $t\in I$. In fact, the function $P_{n}:X\rightarrow X_{n}$ is linear and bounded, hence the functional $\mathbb{P}_{n}:L^2(I,X)\rightarrow L^2(I,X)$ is defined by 
\begin{equation}
    (\mathbb{P}_{n}f)(t) = P_{n}f(t), ~ t\in I
\end{equation}
is also linear and bounded. Therefore, for every $f^*$ in the dual of $L^2(I,X)$ the map $\Psi= f^*\circ\mathbb{P}_{n}$ is in the dual of $L^2(I,X)$ too and 
\begin{align*}
    f^*\circ(\mathbb{P}_{n}(f_{m})) = \Psi(f_{m})\rightarrow \Psi(f) = f^*\circ(\mathbb{P}_{n}(f)), ~ \text{as} ~ m\rightarrow \infty.
\end{align*}
This proves that 
\begin{align*}
    h_{m}= \mathbb{P}_{n}(f_{m})\rightharpoonup \mathbb{P}_{n}(f) = \tilde{f}\in L^2(I,X) ~ \text{as} ~ m\rightarrow \infty
\end{align*}
Therefore, utilizing Lemma \ref{GG} together with Corollay \ref{Colo6.6} (see Appendix section), we can conclude that: 
\begin{equation}\label{44}
    \int ^{t}_{0}(t-s)^{\alpha-1}P_{n}\mathcal{T}_{\alpha}(t-s)P_{n}f_{m}(s)ds \rightarrow \int^{t}_{0}(t-s)^{\alpha-1}P_{n}\mathcal{T}_{\alpha}(t-s)P_{n}f(s)ds ~\text{as} ~ m\rightarrow \infty
\end{equation}
uniformly in $C(I,X)$.\\
Therefore, using \eqref{42} and \eqref{44}, passing limit in \eqref{EQE} we obtain 
\begin{equation*}
    y_{m}(t)\rightarrow y(t)= P_{n}\mathcal{S}_{\alpha}(t)(x_{0}+w) + \int^{t}_{0}(t-s)^{\alpha-1}P_{n}\mathcal{T}_{\alpha}(t-s)P_{n}f(s)ds  
    \end{equation*}
    \begin{equation}
        + \int^{t}_{0}(t-s)^{\alpha-1}P_{n}\mathcal{T}_{\alpha}(t-s)Bu(s)ds, ~ t\in I, 
    \end{equation}
    uniformly in $C(I,X_{n})$, implying that the set $\Lambda_{n}(B^{N_0}_{n})$ is relatively compact in $C(I,X_{n})$.\\
    Define 
    \begin{equation}
        \mathcal{Q}_{n}=\overline{conv}\Lambda_{n}(B^{N_0}_{n})
    \end{equation}
    where $\overline{conv}$ stands for the closed convex hull of a set. Then by Theorem \ref{2.13} the set $\mathcal{Q}_{n}$ is compact in $C(I,X_{n})$ and
\begin{equation}
    \Lambda_{n}(\mathcal{Q}_{n})\subset\mathcal{Q}_{n}
\end{equation}
\textbf{\textit{Step-(iii)}:} It remains to show that the multimap $\Lambda_{n}$ is upper semicontinuous. Observe that the multimap $\Lambda_{n}:\mathcal{Q}_{n}\multimap \mathcal{Q}_{n}$ is locally compact, that means $\Lambda_{n}$ is compact around some neighbourhood of every point. Therefore, by virtue of the Proposition \ref{prop2.10} it is sufficient to prove that the multimap
$\Lambda_{n}$ has a closed graph. For this consider $\{q_{m}\}_{m\in \mathbb{N}}\subset \mathcal{Q}_{n}$, $\{y_{m}\}_{m\in \mathbb{N}}$ be such that $y_{m}\in \Lambda_{n}(q_{m})$ for all $m\in \mathbb{N}$ and $y_{m}\rightarrow y$, $q_{m}\rightarrow q$ in $C(I,X_{n})$. we prove $y\in \Lambda_{n}(q)$. The fact that $y_{m}\in \Lambda_{n}(q_{m})$ implies 
\begin{equation*}
    y_{m}(t) = P_{n}\mathcal{S}_{\alpha}(x_{0}+w_{m}) + \int^{t}_{0}(t-s)^{\alpha-1}P_{n}\mathcal{T}_{\alpha}(t-s)P_{n}f_{m}(s)ds 
\end{equation*}
\begin{equation}\label{EQR}
    + \int^{t}_{0}(t-s)^{\alpha-1}\mathcal{T}_{\alpha}(t-s)Bu(s)ds, ~ t\in I, 
\end{equation}
where $w_{m}\in g(q_{m})$ and $f_{m}\in S_{F}(q_{m})$ for all $m\in \mathbb{N}$. In a similar proof as in \textbf{\textit{Step-(ii)}} we can show
that $w_{m}\rightharpoonup w$ in $X$ up to a subsequence and $f_{m}\rightharpoonup f$ in $L^2(I,X)$ up to a subsequence. Then passing
limit in \eqref{EQR} as $m\rightarrow \infty$ we obtain
\begin{equation*}
    y_{m}(t) \rightarrow y(t) = P_{n}\mathcal{S}_{\alpha}(t)(x_0+w)+\int^{t}_{0}(t-s)^{\alpha-1}P_{n}\mathcal{T}_{\alpha}(t-s)P_{n}f(s)ds 
\end{equation*}
\begin{equation}
    + \int^{t}_{0}P_{n}\mathcal{T}_{\alpha}(t-s)Bu(s)ds, ~ t\in I
\end{equation}
Observe that, by Hypothesis \textbf{(g2)}, the multimap $g$ has a weakly weakly closed graph. Therefore, it follows from $w_{m}\in g(q_{m})$ with $w_{m}\rightharpoonup w$ and $q_{m}\rightarrow q$ that $w\in g(q)$. Moreover, by Theorem \ref{SFT} we conclude that $f\in S_{F}(q)$. Thus $y\in \Lambda_{n}(q)$ and hence the multimap $\Lambda_{n}$ has a closed graph.\\

\textbf{\textit{Step-(iv)}:} Thus, following \textbf{\textit{Step-(i) - Step-(iii)}}, it follows that the multimap $\Lambda_{n}:\mathcal{Q}_{n}\multimap \mathcal{Q}_{n}$ is
upper semicontinuous, where $\mathcal{Q}_{n}\subset C(I,X_{n})$ is compact, convex. Moreover, it is not difficult to show that the multimap $\Lambda_{n}$ has closed and convex values. 
Hence, the multimap $\Lambda_{n}$ satisfies all the conditions of Glicksberg Ky Fan Fixed Point Theorem \ref{fixed}. Therefore, by Theorem \ref{fixed} we have a fixed point of the map $\Lambda_{n}$ for each $n\in \mathbb{N}$, say $q_{n}\in \mathcal{Q}_{n}$. The definition of the map $\Lambda_{n}$ gives
\begin{equation*}
    q_{n}(t) = P_{n}\mathcal{S}_{\alpha}(t)(x_0+w_{n}) + \int^{t}_{0}(t-s)^{\alpha-1}P_{n}\mathcal{T}_{\alpha}(t-s)P_{n}f(s)ds 
\end{equation*}
\begin{equation}\label{EQo}
   + \int^{t}_{0}(t-s)^{\alpha-1}P_{n}\mathcal{T}_{\alpha}(t-s)Bu(s)ds, ~ t\in I 
\end{equation}
In the above $w_{n}\in g(q_{n})$, $f_{n}\in S_{F}(q_{n})$ for all $n\in \mathbb{N}$. We prove that there exists a subsequence $\{q_{n_k}\}_{k\in \mathbb{N}}$ of $\{q_{n}\}_{n\in \mathbb{N}}$ such that, $q_{n_k}$ converges weakly to $q\in C(I,X)$ where 
\begin{equation*}\label{52}
    q(t)= \mathcal{S}_{\alpha}(t)(x_0+w)+\int^{t}_{0}(t-s)^{\alpha-1}\mathcal{T}_{\alpha}(t-s)f(s)ds 
\end{equation*}
\begin{equation}
    +\int^{t}_{0}(t-s)^{\alpha-1}\mathcal{T}_{\alpha}(t-s)Bu(s)ds, ~ \text{for every} ~ t\in I, 
\end{equation}
 where $w\in g(q), f\in S_{F}(q)$. Considering hypothesis (\textbf{g1}), we observe that $\{w_{n}\}_{n\in\mathbb{N}}\subset X$ is bounded
in $X$ and $X$ is reflexive, so there exists $w\in X$ such that $w_{n}\rightharpoonup w$ in $X$. According to the reflexivity of $X$ together with Proposition \ref{Prop2} we have 
\begin{equation}\label{53}
    P_{n}\mathcal{S}_{\alpha}(t)(x_0+w_n)\rightharpoonup \mathcal{S}_{\alpha}(t)(x_0+w) ~ \text{in} ~ X
\end{equation}
Additionally, based on Hypothesis (\textbf{F2}), the sequence $\{f_{n}\}_{n\in\mathbb{N}}$ is integrably bounded and hence $\{f_{n}\}_{n\in \mathbb{N}}\subset L^2(I,X)$ is weakly relatively compact in $L^2(I,X)$. Therefore, there exists
$f\in L^2(I,X)$ such that $f_{n}\rightharpoonup f$ in $L^2(I,X)$. In view of Lemma \ref{Lem6.3}
(can be found in Appendix
section) we obtain $\mathbb{P}_{n}f_{n}\rightharpoonup f$ in $L^2(I,X)$, where $\mathbb{P}_{n}$ is given by \eqref{117}. Therefore, by Theorem \ref{Thm6.1}
we obtain
\begin{equation}\label{54}
    \int^{t}_{0}(t-s)^{\alpha-1}P_{n}\mathcal{T}_{\alpha}(t-s)P_{n}f_{n}(s)ds\rightharpoonup \int^{t}_{0}(t-s)^{\alpha-1}\mathcal{T}_{\alpha}(t-s)f(s)ds
\end{equation}
Combining equations \eqref{53}, \eqref{54} together we obtain from \eqref{EQo}
\begin{equation*}
    q_{n}(t)\rightharpoonup q(t)= \mathcal{S}_{\alpha}(t)(x_0+w)+\int^{t}_{0}(t-s)^{\alpha-1}\mathcal{T}_{\alpha}(t-s)f(s)ds 
\end{equation*}
\begin{equation}
    +\int^{t}_{0}(t-s)^{\alpha-1}\mathcal{T}_{\alpha}(t-s)Bu(s)ds, ~ t\in I.
\end{equation}
The fact that $q_{n}(t)\rightharpoonup q(t)$ in $X$ for every $t\in I$ and $\{q_{n}\}_{n\in \mathbb{N}}\subset C(I,X)$ is bounded, we conclude that
$q_{n}\rightharpoonup q$ in $C(I,X)$. As $w_{n}\rightharpoonup w$ in $X$, and $w_{n}\in g(q_{n})$, we can use hypothesis (\textbf{g2}) to conclude that $w\in g(q)$. Additionally, since $f_{n}\rightharpoonup f$ in $L^2(I,X)$, $q_{n}(t)\rightharpoonup q(t)$ in $X$, and $f_{n}\in S_{F}(q_n)$, we can apply Theorem \ref{SFT} to establish that $f\in S_{F}(q)$. In summary, $q$ is a mild solution to the problem \eqref{1.1}.

\end{proof}
\section{Partially Null Controllability:}This section is dedicated to examining the $L^p$- partially null controllability of the fractional-order control problem \eqref{1.1}. Here control space $U$ is a uniformly convex Banach space. This section is organized into three subsections. \\
\begin{itemize}
    \item (i) In First subsection, we construct the control to drive system \eqref{1.1} from $x_{0}$ to $0$.
    \item (ii) In second subsection, we sketch the idea and novelty of the proof.
    \item (iii) In third subsection, we applies this control to examine the partially null controllability of problem \eqref{1.1}.
\end{itemize}

\subsection{Construction of the Control}\label{Sub4.1}
The goal of this subsection is to construct a control that directs the solutions of system \eqref{1.1}  from any initial state to the zero state. We consider operator $\mathcal{W}: L^p(I, U)\to X$ which is defined as in \eqref{ControlW}. Observe that the operator $\mathcal{W}$ is a bounded linear operator, though it may not be surjective. \\
 
Consider the fractional-order linear control system with additive source term
\begin{equation}\label{L1}
	\begin{cases}
		^{C}D^{\alpha}_tq(t)=\mathbb{A}q(t)+\mathbb{B}u(t)+f(t), ~ t\in I;\\
		q(0)=x_0.
	\end{cases}
\end{equation}
We define $\mathcal{Z}: X\times L^2(I,X)\to X$ by
\begin{equation*}
	\mathcal{Z}(x_0,f)=\mathcal{S}_{\alpha}(\nu)x_0+\int_{0}^{\nu}(\nu-s)^{\alpha-1}\mathcal{T}_{\alpha}(\nu-s)f(s)ds, ~ x_0\in X, ~ f\in L^2(I,X).
\end{equation*}
\begin{definition}\label{DD4.1}
    We say that the linear fractional-order control system with source term \eqref{L1} is partially null controllable if for any initial state $x_{0}\in X$, there exists a mild solution $q\in C(I,X)$ of \eqref{1.1} corresponding to a suitable control $u\in L^{p}(I,U)$ such that $q(\nu)= 0$, that means
    \begin{equation}
		q(\nu)=\mathcal{S}_{\alpha}(\nu)x_0+\int_{0}^{\nu}(\nu-s)^{\alpha-1}\mathcal{T}_{\alpha}(\nu-s)f(s)ds+\int_{0}^{\nu}(\nu-s)^{\alpha-1}\mathcal{T}_{\alpha}(\nu-s)\mathbb{B}u(s)ds=0,
	\end{equation}
\end{definition}

The following proposition provides an alternate condition for partially null controllability of the system \eqref{L1}.
\begin{proposition}
	The system \eqref{L1} is partially null controllable on $I$ if $Im(\mathcal{Z})\subset ~Im(\mathcal{W})$.	
\end{proposition}
Consider the fractional-order linear control system with source term represented by equation \eqref{L1} is partially null controllable in $I$. Define the kernel of an operator $\mathcal{W}$.
\begin{equation*}
	Ker~\mathcal{W}=\{u\in L^p(I,U): \mathcal{W}(u)=0\}.
\end{equation*}
As $\mathcal{W}$ is bounded and linear, the kernel $Ker(\mathcal{W})$ is a closed subspace of $L^p(I,U)$, and the quotient space $L^p(I,U)/Ker~\mathcal{W}$ is a Banach space with the induced norm  
 
\begin{equation*}
	\norm{[u]}_0=\inf_{v\in [u]}\norm{v}, [u]=\{v\in L^p(I,U): \mathcal{W}(v)=\mathcal{W}(u)\}=u+Ker~\mathcal{W}.
\end{equation*}
Define $\tilde{\mathcal{W}}:L^p(I,U)/Ker~\mathcal{W}\to X$ by
\begin{equation*}
	\tilde{\mathcal{W}}([u])=\mathcal{W}(u), [u]\in L^p(I,U)/ Ker~(\mathcal{W}).
\end{equation*}
Thus, $\tilde{\mathcal{W}}$ is well-defined, linear and injective. Additionaly, $\tilde{\mathcal{W}}$ is bounded, and we conclude that $\norm{\tilde{\mathcal{W}}}=\norm{\mathcal{W}}$. 
 
Next we define $\Pi: L^p(I,U)/Ker~\mathcal{W}\to L^p(I,U)$ such that 
\begin{equation*}
	\mathcal{W}(\Pi([u]))=\mathcal{W}(u),~\text{and}~ \norm{\Pi([u])}=\min\{\norm{v}: \mathcal{W}(u)=\mathcal{W}(v)\}.
\end{equation*} 
Clearly, $\Pi$ is continuous (see \cite{malaguti2021begin}). We further define
\begin{equation}\label{Control}
	\mathcal{W}^{-1}=\Pi\circ \tilde{\mathcal{W}}^{-1}:~Im~(\mathcal{W})\to L^p(I,U).
\end{equation}
One can easily observe that $\mathcal{W}\circ \mathcal{W}^{-1}(x)=x$ for all $x\in Im~(\mathcal{W})$, with $\mathcal{W}^{-1}$ being continuous, though it is not necessarily linear (see \cite{malaguti2019exact} and \cite{malaguti2021begin} ).

We construct $u\in L^p(I,U)$ as 
\begin{equation*}
	u=-\mathcal{W}^{-1}\left[\mathcal{S}_{\alpha}(\nu)x_0+\int_{0}^{\nu}(\nu-s)^{\alpha-1}\mathcal{T}_{\alpha}(\nu-s)f(s)ds\right].
\end{equation*}
The constructed control ensures the transferring  of the system \eqref{L1} from the initial state $x_0$ to the final state $0$. Consequently, we propose the following proposition.

\begin{proposition}
	Suppose the system \eqref{L1} is partially null controllable on $I$. Then the control
	\begin{equation}\label{ct3.4}
		u=-\mathcal{W}^{-1}\left[\mathcal{S}_{\alpha}(\nu)x_0+\int_{0}^{\nu}(\nu-s)^{\alpha-1}\mathcal{T}_{\alpha}(\nu-s)f(s)ds\right],
	\end{equation} 
	transfers the system \eqref{L1} from $x_0$ to $0$.
\end{proposition} 
\begin{remark}
According to Definition \ref{DD4.1}, the control $u$ in \eqref{ct3.4} is well-defined because $Im(\mathcal{Z})\subset Im(\mathcal{W})$.
\end{remark}
\begin{remark}
If the map $\mathcal{W}$ specified in \eqref{ControlW} is surjective, one can define the inverse map $\mathcal{W}^{-1}: X\to L^p(I,U)$, which is the right inverse of the map $\mathcal{W}$.  
	
\end{remark}
\begin{remark}
For a separable Hilbert space $U$, the mapping $\mathcal{W}^{-1}$ is both linear and bounded, as analyzed in \cite{MR2098245}.
	
\end{remark}

\subsection{\textbf{Idea and Novelty of the proof}} This subsection highlighting the idea and novelty of the proof. The standard approach used to establish the partially null controllability result is as follows. For each $n\in \mathbb{N}$, define the map $\Sigma_{n}: C(I,X_{n})\multimap C(I,X_{n})$ as follows: for $q\in C(I,X_{n})$ the set $\Sigma_{n}(q)$ consists all those $y\in C(I,X_{n})$ such that 
\begin{equation}
    y(t) = P_{n}\mathcal{S}_{\alpha}(t)(x_{0}+w)+\int^{t}_{0}(t-s)^{\alpha-1}P_{n}\mathcal{T}_{\alpha}(t-s)P_{n}f(s)ds +\int^{t}_{0}(t-s)^{\alpha-1}P_{n}\mathcal{T}_{\alpha}(t-s)\mathbb{B}u(s)ds,
\end{equation}
$t\in I$, where $w\in g(q), f\in S_{F}(q) ~\text{and} ~ u\in L^{p}(I,U)$ is given by 
\begin{equation}\label{Contrr}
    u=-\mathcal{W}^{-1}\mathcal{Z}_{n}(w,f)
\end{equation}
where 
\begin{equation}
    \mathcal{Z}_{n}(w,f) = \mathcal{S}_{\alpha}(\nu)P_{n}(x_{0}+w)+\int^{\nu}_{0}(\nu-s)^{\alpha-1}P_{n}\mathcal{T}_{\alpha}(\nu-s)P_{n}f(s)ds
\end{equation}

Where, $P_{n}:X\rightarrow X_{n}$ is the natural projection maps defined in \eqref{Proje}. As presented in equation \eqref{Contrr}, $\mathcal{W}^{-1}$ as the right inverse of the map $\mathcal{W}:L^{p}(I,U)\rightarrow X$. In a uniformly convex Banach space $U$, the mapping $\mathcal{W}^{-1}$ is continuous but not necessarily linear. Consequently, the nonlinearity in the control term \eqref{Contrr} leads to the non-convexity of the set $\Sigma_{n}(q)$. Based on current knowledge, most multivalued fixed point theorems guaranteeing existence results demand that the fixed point map takes convex values. Consequently, these fixed-point theorems cannot be utilized to demonstrate the existence of fixed points for the map $\Sigma_{n}$. To handle this issue, for every $n\in\mathbb{N}$, we fix an element $w_{n}\in X$ and define a map $ \Upsilon_{n}: L^2(I,X)\rightarrow C(I,X)$ as follows: for $f\in L^2(I,X)$ we define
\begin{equation*}
    (\Upsilon_{n}f)(t)=P_{n}\mathcal{S}_{\alpha}(t)(x_{0}+w_{n})+\int^{t}_{0}(t-s)^{\alpha-1}P_{n}\mathcal{T}_{\alpha}(t-s)P_{n}f(s)ds
    \end{equation*}
    \begin{equation}
     +\int^{t}_{0}(t-s)^{\alpha-1}P_{n}\mathcal{T}_{\alpha}(t-s)\mathbb{B}u(s)ds, ~ t\in I
\end{equation}
where $u\in L^{p}(I,U)$ is given by 
\begin{equation}
    u= - \mathcal{W}^{-1}\mathcal{Z}_{n}(w_{n},f)
\end{equation}
where
\begin{equation}
    \mathcal{Z}_{n}(w_{n},f) = \mathcal{S}_{\alpha}(\nu)P_{n}(x_{0}+w_{n})+\int^{\nu}_{0}(\nu-s)^{\alpha-1}\mathcal{T}_{\alpha}(\nu-s)P_{n}f(s)ds
\end{equation}
We establish the continuity of the map $\Upsilon_{n}$ from $w-L^2(I,X)$ into $C(I,X)$. Subsequently, we consider the map $S_{F}: C(I,X)\multimap L^2(I,X)$ as defined in equation \eqref{selection}.  Then define $\Gamma_{n} : L^2(I,X)\multimap L^2(I,X)$ as 
\begin{equation}
    \Gamma_{n}(f)=S_{F}(\Upsilon_{n}(f)), ~ f\in L^2(I,X)
\end{equation}
Given that the multimap $\Gamma_{n}$ has convex values. Therefore, by introducing this map, we can resolve the problem of convexity by defining this map.\\
The partially null controllability proof for problem \eqref{1.1} can be broken down into the following steps:

\textbf{\textit{Step-(i)}} We prove that for each $n\in \mathbb{N}$, $\Gamma_{n}$ possesses a fixed point $f_{n}\in L^2(I,X)$. Based on the definition of the multimap $\Gamma_{n}$, it follows that $f_{n}\in S_{F}(\Upsilon_{n}(f_{n})), ~ n\in \mathbb{N}$. Define $q_{n}=\Upsilon_{n}(f_{n})$. As a result, $f_{n}\in S_{F}(q_{n})$, and according to the definition of the map $\Upsilon_{n}$, it can be written as 
\begin{equation*}
    q_{n}(t) = P_{n}\mathcal{S}_{\alpha}(t)(x_{0}+w_{n})+\int^{t}_{0}(t-s)^{\alpha-1}P_{n}\mathcal{T}_{\alpha}(t-s)P_{n}f_{n}(s)ds
    \end{equation*}
    \begin{equation}
    +\int^{t}_{0}(t-s)^{\alpha-1}P_{n}\mathcal{T}_{\alpha}(t-s)\mathbb{B}u_{n}(s)ds, ~ t\in I, 
\end{equation}
where 
\begin{equation}
    u_{n} =-\mathcal{W}^{-1}\mathcal{Z}_{n}(w_{n},f_{n}),
\end{equation}
 and 
 \begin{equation}
     \mathcal{Z}_{n}(w_{n},f_{n}) = \mathcal{S}_{\alpha}(\nu)P_{n}(x_{0}+w_{n})+\int^{\nu}_{0}(\nu-s)^{\alpha-1}P_{n}\mathcal{T}_{\alpha}(\nu-s)P_{n}f_{n}(s)ds
 \end{equation}

 \textbf{\textit{Step-(ii)}} In \textbf{\textit{Step-(i)}}, the sequence $\{q_{n}\}_{n\in\mathbb{N}}$ contains a subsequence that converges weakly to a continuous mapping $q:I\rightarrow X$ such that
 \begin{equation}
     q(t) = \mathcal{S}_{\alpha}(t)(x_{0}+w)+\int^{t}_{0}(t-s)^{\alpha-1}\mathcal{T}_{\alpha}(t-s)f(s)ds+\int^{t}_{0}(t-s)^{\alpha-1}\mathcal{T}_{\alpha}(t-s)\mathbb{B}u(s)ds, ~ t\in I,
 \end{equation}
 where $w\in g(q), ~ f\in S_{F}(q)$ and satisfies $q(\nu)=0$.

\subsection{Main Result-Partially Null Control}
The $L^p$-partially null controllability of the problem \eqref{1.1} is explored in this subsection, based on the control introduced in Subsection \eqref{Sub4.1}. The next theorem establishes this controllability property.

\begin{theorem}\label{MR}
Suppose Hypotheses (\textbf{T}) and (\textbf{U}) hold true. Further, consider that the multivalued nonlinearity $F$ fulfills conditions \textbf{(F1)}-\textbf{(F4)}  and the nonlocal multifunction $g$ meets assumptions (\textbf{g1})-(\textbf{g2}). Under these assumptions, system  \eqref{1.1} achieves partially null controllability over the interval $I$. 
\end{theorem}
\begin{proof}
	Let us define three sets as given follows: 
	\begin{equation}
		B_N=\{x\in X: \norm{x}\le N\}, N>0,
	\end{equation}
	\begin{equation}
		B^N_C=\{q\in C(I,X): \norm{q(t)}\le N, t\in I\}, N>0,
	\end{equation}
	and 
	\begin{equation}
		Q^N_{F}=\{f(\cdot)\in L^2(I, X): \norm{f(t)}\le \eta_N(t)~\text{a.a.}~t\in I\},
	\end{equation}
	where $\eta_N(\cdot)$ is as in Hypothesis \textbf{(F2)} corresponding to the set $B_N$.
	Let $w_n\in g(B_C^N)\subset X$ be fixed for each $n\in \mathbb{N}$. The map $\Upsilon_n: Q^N_{F}\to C(I,X)$ is defined in the following way:
	For $f\in Q^N_{F}, \Upsilon_n (f)$ satisfies 
	\begin{equation*}
		(\Upsilon_n f)(t)=P_n\mathcal{S}_{\alpha}(t)(x_0+w_n)+\int_{0}^{t}(t-s)^{\alpha-1}P_n\mathcal{T}_{\alpha}(t-s)P_nf(s)ds
	\end{equation*} 
    \begin{equation}\label{A1}
        +\int_{0}^{t}(t-s)^{\alpha-1}P_n \mathcal{T}_{\alpha}(t-s)\mathbb{B}u(s)ds, t\in I,
    \end{equation}
	where $u\in L^p(I,U)$ is given by
	\begin{equation}\label{A2}
		u=-\mathcal{W}^{-1}\mathcal{Z}_n(w_n,f),
	\end{equation}
	where 
	\begin{equation}\label{A3}
		\mathcal{Z}_n(w_n,f)=\left[\mathcal{S}_{\alpha}(\nu)P_n(x_0+w_n)+\int_{0}^{\nu}(\nu-s)^{\alpha-1}\mathcal{T}_{\alpha}(\nu-s)P_nf(s)ds\right].
	\end{equation}
	The proof is concluded step by step. \\

	\textbf{\textit{Step-(i):}} We establish that the multimap $\Upsilon_n$ maps $Q^{N_0}_{F}$ into $B_C^{N_0}$ for some $N_0>0$ for each $n\in \mathbb{N}$. Let $y=\Upsilon_n(f), f\in Q^N_{F}$. According to the definition of the map $\Upsilon_n$, it follows that $y$ satisfies \eqref{A1}-\eqref{A3}.  We will now proceed with the computation by using hypothese  ($\textbf{F2}$) and H$\ddot{\text{o}}$lder's inequality,
	\begin{align}\label{3.3a}
		\norm{\mathcal{Z}_n(w_n,f)}\le M\norm{x_0+w_n}+\frac{M}{\Gamma(\alpha)}\kappa_{1}\norm{\eta_N}_{L^{\frac{1}{\alpha_1}}(I;\mathbb{R}^+)},
	\end{align}
	and consequently
	\begin{align*}	\norm{\mathcal{W}^{-1}\mathcal{Z}_n(w_n,f)}_{L^p(I,U)}=&\norm{\Pi\circ\tilde{\mathcal{W}}^{-1}\mathcal{Z}_n(w_n,f)}_{L^p(I,U)}
		=\norm{\tilde{\mathcal{W}}^{-1}\mathcal{Z}_n(w_n,f)}.
	\end{align*}
	Using the expression \eqref{3.3a} we obtain
	\begin{equation*}
		\norm{\mathcal{W}^{-1}\mathcal{Z}_n(w_n,f)}_{L^p(I,U)}\le \norm{\tilde{\mathcal{W}}^{-1}}\left[M\norm{x_0+w_n}+\frac{M}{\Gamma(\alpha)}\kappa_{1}\norm{\eta_N}_{L^\frac{1}{\alpha_1}(I;\mathbb{R}^+)}\right].
	\end{equation*}
	By the structure of the set $Q_F^N$ we have from \eqref{A1} and \eqref{A2},
	\begin{align*}
		\norm{y(t)}\le& \norm{P_n\mathcal{S}_{\alpha}(t)(x_0+w_n)}+\int_{0}^{t}(t-s)^{\alpha-1}\norm{P_n\mathcal{T}_{\alpha}(t-s)P_nf(s)}ds\\
        &+\int_{0}^{t}(t-s)^{\alpha-1}\norm{P_n\mathcal{T}_{\alpha}(t-s)\mathbb{B}u(s)}ds\\
		\le&M\norm{x_0+w_n}+\frac{M}{\Gamma(\alpha)}\kappa_{1}\norm{\eta_N}_{L^{\frac{1}{\alpha_1}}(I;\mathbb{R}^+)}\\
		&+\frac{M}{\Gamma(\alpha)}\norm{\mathbb{B}}\kappa_2\norm{\mathcal{W}^{-1}\mathcal{Z}_n(w_n,f)}_{L^p(I;U)}.
	\end{align*}
In the last inequality we use H$\ddot{\text{o}}$lder's inequality and $\frac{1}{p}+\frac{1}{p^{\prime}}=1$.\\
After substituting the value of $\norm{\mathcal{W}^{-1}\mathcal{Z}_n(w_n,f)}_{L^p(I,U)}$ into the expression above and simplifying, we arrive at:
	\begin{align*}
		\norm{y(t)}
		%\le&M\norm{x_0}+M\norm{\phi_N}_{L^1([0,a];\mathbb{R}^+)}+M\norm{B}a^{\frac{1}{p^{\prime}}}\\
		%	&\hspace{5cm}\times\norm{\tilde{W}^{-1}}[M\norm{x_0}+M\norm{\phi_N}_{L^1([0,a];\mathbb{R}^+)}]\\
		\le&M\norm{x_0}+M\norm{g(B_C^N)}\\
        &+\frac{M}{\Gamma(\alpha)}\norm{\mathbb{B}}\kappa_2\norm{\tilde{\mathcal{W}}^{-1}}\left[M\norm{x_0}+M\norm{g(B_C^N)}\right]\\
		&+\frac{M}{\Gamma(\alpha)}\kappa_{1}\left[1+\frac{M}{\Gamma(\alpha)}\norm{\mathbb{B}}\kappa_2\norm{\tilde{\mathcal{W}}^{-1}}\right]\norm{\eta_N}_{L^{\frac{1}{\alpha_1}}(I;\mathbb{R}^+)}\\
		\le&D_1+D_2\norm{\eta_N}_{L^{\frac{1}{\alpha_1}}(I;\mathbb{R}^+)}+D_3\norm{g(B_C^N)},
	\end{align*}
	where
	\begin{align*}
		D_1=&M\norm{x_0}\left[1+\frac{M}{\Gamma(\alpha)}\norm{\mathbb{B}}\kappa_{2}\norm{\tilde{W}^{-1}}\right];\\
		D_2=&\frac{M}{\Gamma(\alpha)}\kappa_{1}\left[1+\frac{M}{\Gamma(\alpha)}\kappa_{2}\norm{\mathbb{B}}\norm{\tilde{\mathcal{W}}^{-1}}\right];\\
		D_3=&M\left[1+\frac{M}{\Gamma(\alpha)}\kappa_{2}\norm{\mathbb{B}}\norm{\tilde{W}^{-1}}\right],
	\end{align*}
	and $$\norm{g(B_C^N)}=\sup\{\norm{w}: w\in g(q), q\in B_C^N\}.$$
	With (Proposition 3 of \cite{malaguti2021begin}) in mind and under the framework of Hypotheses (\textbf{F2}) and (\textbf{g1}) we conclude that
	\begin{equation}
		\liminf_{N\to \infty} \frac{1}{N}\norm{\eta_N}_{L^{\textcolor{red}{\frac{1}{\alpha_1}}}(I,\mathbb{R}^+)}=0~\text{and}~\lim_{N\to \infty} \frac{\norm{g(B_C^N)}}{N}=0.
	\end{equation}
	Considering this observation, we can conclude that
	\begin{equation}
		0\le \liminf_{N\to \infty}\frac{1}{N}\norm{\Upsilon_n(Q_F^N)}\le \frac{D_1+D_2\norm{\eta_N}_{L^\frac{1}{\alpha_1}(I,\mathbb{R}^+)}+D_3\norm{g(B_C^N)}}{N}=0.
	\end{equation}
	This implies
	\begin{equation}
		\liminf_{N\to \infty}\frac{1}{N}\norm{\Upsilon_n(Q_F^N)}=0.
	\end{equation}
	To complete the proof, it suffices to verify that there exists $N_0\in \mathbb{N}$ such that $\Upsilon_n(Q_F^{N_0})\subset B_C^{N_0}$ for every $n\in \mathbb{N}$. Suppose, by contradiction, that for every $N\in \mathbb{N}$ there exists $\bar{n}\in \mathbb{N}$ and $f_N\in Q_F^{N}$ such that $\Upsilon_{\bar{n}}(f_N)\notin B_C^{N}$. By the fact that  $\Upsilon_{\bar{n}}(f_N)\notin B_C^{N}$ we have
	\begin{equation}
		N<\norm{\Upsilon_{\bar{n}}(f_N)}<D_1+D_2 \norm{\eta_N}_{L^{\frac{1}{\alpha_1}}(I,\mathbb{R}^+)}+D_3 \norm{g(B_C^N)},
	\end{equation}
	which implies
	\begin{equation}
		1<\frac{1}{N}\norm{\Upsilon_{\bar{n}}(f_N)}<\frac{D_1+D_2 \norm{\eta_N}_{L^{\frac{1}{\alpha_1}}(I,\mathbb{R}^+)}+D_3 \norm{g(B_C^N)}}{N}.
	\end{equation}
	Taking the limit as $N\to \infty$ leads to a contradiction,  
	which demonstrates that the map $\Upsilon_{n}$ maps $Q_F^{N_0}$ into $B_C^{N_0}$ for each $n\in \mathbb{N}$.\\

	\textbf{\textit{Step-(ii):}} We verify the continuity of the mapping $\Upsilon_n$ from $w-Q^{N_0}_{F}$ into $C(I,X)$.  Let $\{f_m\}_{m\in \mathbb{N}}\subset Q^{N_0}_{F}$ be such that $f_m\rightharpoonup f$ in $L^2(I, X)$. Let $y_m=\Upsilon_n(f_m)$, then the following equation holds for $t\in I$:
	\begin{equation*}\label{3.4e}
		y_m(t)=P_n\mathcal{S}_{\alpha}(t)(x_0+w_n)+\int_{0}^{t}(t-s)^{\alpha-1}P_n\mathcal{T}_{\alpha}(t-s)P_nf_m(s)ds
	\end{equation*}
    \begin{equation}
        +\int_{0}^{t}(t-s)^{\alpha-1}P_n \mathcal{T}_{\alpha}(t-s)\mathbb{B}u_m(s)ds,
    \end{equation}
	$ t\in I,$ where $u_m\in L^p(I,U)$ is given by
	\begin{equation}\label{3.5}
		u_m=-\mathcal{W}^{-1}\mathcal{Z}_n(w_n,f_m),
	\end{equation}
	where
	\begin{equation}\label{3.61}
		\mathcal{Z}_n(w_n,f_m)=\mathcal{S}_{\alpha}(\nu)P_n(x_0+w_n)+\int_{0}^{\nu}(\nu-s)^{\alpha-1}\mathcal{T}_{\alpha}(\nu-s)P_nf_m(s)ds.
	\end{equation}
	We consider $h_m\in L^2(I, X)$ defined as $h_m(t)=P_nf_m(t)$ for $t\in I$, and $m\in \mathbb{N}$.
    The set $\{h_m(t): m\in \mathbb{N}\}$ is bounded for a.a $t\in I$ and is entirely contained in the finite-dimensional space $X_n$. As a result, it is relatively compact in both $X_n$ and $X$. Furthermore, since $f_m\in Q_F^{N_0}$, the sequence $\{h_m\}_{m\in \mathbb{N}}$ is uniformly integrable. According to the definition \ref{semicompact} $\{h_m\}_{m\in \mathbb{N}}$ is a semicompact sequence. Additionally, $h_m$ weakly converges to $\tilde{f}$ in $L^2(I, X)$, where $\tilde{f}(t)=P_nf(t),$ a.a. $t\in I$. Moreover, the functional $\mathbb{P}_n:L^2(I, X)\to L^2(I, X)$  is defined by
	\begin{equation}\label{Projection}
		(\mathbb{P}_nf)(t)=P_nf(t), t\in I.
	\end{equation}
It is a linear and bounded map because the function  $P_n:X\to X_n$ is linear and bounded. 
Therefore, for every $f^*$ in the dual of $L^2(I, X)$ the map $\Psi=f^* \circ\mathbb{P}_n$ is in the dual of $L^2(I, X)$ too and
	\begin{equation*}
		f^*\circ(\mathbb{P}_n(f_m))=\Psi(f_m)\to \Psi(f)=f^*\circ(\mathbb{P}_n(f)), ~\text{as}~ m\to \infty.
	\end{equation*}
	This proves that 
	\begin{equation*}
		h_m=\mathbb{P}_n(f_m)\rightharpoonup \mathbb{P}_n(f)=\tilde{f} ~\text{in}~ L^2(I, X)~\text{as}~m\to \infty.
	\end{equation*}
	By invoking Corrolary \eqref{Colo6.6} (see Appendix section) and equation \eqref{2.3}, we arrive at the conclusion that:
	\begin{equation}\label{3.8c}
		\int_{0}^{t}(t-s)^{\alpha-1}P_n\mathcal{T}_{\alpha}(t-s)P_nf_m(s)ds\to \int_{0}^{t}(t-s)^{\alpha-1}P_n\mathcal{T}_{\alpha}(t-s)P_nf(s)~\text{in}~X~\text{as}~m\to \infty,
	\end{equation}
	uniformly in $C(I,X_{n})$.
	
	By employing Theorem \eqref{Thm6.1} outlined in the Appendix, we can obtain
	\begin{equation}\label{3.9b}
		\int_{0}^{\nu}(\nu-s)^{\alpha-1}\mathcal{T}_{\alpha}(\nu-s)P_nf_m(s)ds\rightharpoonup \int_{0}^{\nu}(\nu-s)^{\alpha-1}\mathcal{T}_{\alpha}(\nu-s)P_nf(s)ds ~\text{in}~ X.
	\end{equation}
	On the basis of \eqref{2.2} together with \eqref{3.9b}  we obtain from \eqref{3.61},
	\begin{equation*}
		\mathcal{Z}_n(w_n,f_m)\to \mathcal{Z}_n(w_n,f)~\text{in}~  X~\text{as}~m\to \infty,
	\end{equation*}
	where $\mathcal{Z}_n(w_n,f)$ is given by \eqref{A3}.
	By the continuity of $\mathcal{W}^{-1}$ we have that
	\begin{equation*}
		u_m= - \mathcal{W}^{-1}(\mathcal{Z}_n(w_n,f_m))\to -\mathcal{W}^{-1} (\mathcal{Z}_n(w_n,f))=u,~\text{in}~ L^p(I,U)~\text{as}~m\to \infty.
	\end{equation*}
	We now estimate
	\begin{align*}
		&\norm{\int_{0}^{t}(t-s)^{\alpha-1}P_n\mathcal{T}_{\alpha}(t-s)\mathbb{B}\mathcal{W}^{-1}\mathcal{Z}_n(w_n,f_m)(s)ds-\int_{0}^{t}(t-s)^{\alpha-1}P_n\mathcal{T}_{\alpha}(t-s)\mathbb{B}\mathcal{W}^{-1}\mathcal{Z}_n(w_n,f)(s)ds}\\
		=&\norm{\int_{0}^{t}(t-s)^{\alpha-1}P_n\mathcal{T}_{\alpha}(t-s)\mathbb{B}(u_m(s)-u(s))ds}\\
		\le &\int_{0}^{t}(t-s)^{\alpha-1}\norm{P_nT(t-s)\mathbb{B}(u_m(s)-u(s))}ds\\
		\le &\frac{M}{\Gamma(\alpha)}\norm{\mathbb{B}}\int_{0}^{\nu}(\nu-s)^{\alpha-1}\norm{u_m(s)-u(s)}ds\\
		\le &\frac{M}{\Gamma(\alpha)}\norm{\mathbb{B}}\kappa_{2}\norm{u_m-u}_{L^p(I, U)}\to 0~\text{as}~m\to \infty.
	\end{align*}
	and therefore we conclude that
	\begin{equation}\label{3.10d}
		\int_{0}^{t}(t-s)^{\alpha-1}P_n\mathcal{T}_{\alpha}(t-s)\mathbb{B}u_m(s)ds\to \int_{0}^{t}(t-s)^{\alpha-1}P_n\mathcal{T}_{\alpha}(t-s)\mathbb{B}u(s)ds~\text{as}~m\to \infty,
	\end{equation}
	uniformly in $t\in I$. Combining \eqref{3.8c} and \eqref{3.10d} together and passing limit in \eqref{3.4e} we obtain
	\begin{equation*}
		y_m(t)\to y(t)=P_n\mathcal{S}_{\alpha}(t)(x_0+w_n)+\int_{0}^{t}(t-s)^{\alpha-1}P_n\mathcal{T}_{\alpha}(t-s)P_nf(s)ds
	\end{equation*}
    \begin{equation*}
        +\int_{0}^{t}(t-s)^{\alpha-1}P_n\mathcal{T}_{\alpha}(t-s)\mathbb{B}u(s)ds ~ \text{as} ~ m\rightarrow \infty
    \end{equation*}
	uniformly in $t\in I$, where
	\begin{equation*}
		u=-\mathcal{W}^{-1}\mathcal{Z}_n(w_n,f),
	\end{equation*}
	where
	\begin{equation}\label{3.6}
		\mathcal{Z}_n(w_n,f)=\mathcal{S}_{\alpha}(\nu)P_n(x_0+w_n)+\int_{0}^{\nu}(\nu-s)^{\alpha-1}\mathcal{T}_{\alpha}(\nu-s)P_nf(s)ds.
	\end{equation}
	This leads to the conclusion that the map $\Upsilon_n$ is continuous.\\

	\textbf{\textit{Step-(iii):}} We define the map $S_{F}:\Upsilon_n(Q^{N_0}_{F})\multimap L^2(I, X)$ by
	\begin{equation}
		S_{F}(q)=\{f\in L^2(I, X): f(t)\in F(t,q(t))~\text{a.a.}~t\in I\}, q\in \Upsilon_n(Q^{N_0}_{F}).
	\end{equation}
	By Theorem (\ref{Thm2.29}) and Theorem \eqref{SFT}, it follows that the multimap $S_{F}$ has nonempty closed, convex values and has a weakly closed graph. \\

	\textbf{\textit{Step-(iv):}} We now define the multimap $\Gamma_n: Q^{N_0}_{F}\multimap L^2(I, X)$ by 
	\begin{equation}
		\Gamma_n(f)=S_{F}(\Upsilon_n(f)), f\in Q^{N_0}_{F}.
	\end{equation}   
	It is evident that the multimap $\Gamma_n$ is clearly well defined, with weakly compact and convex values. Moreover, the map $\Gamma_n$ maps $Q^{N_0}_{F}$ into itself. Infact if $\xi\in \Gamma_n(f), f\in Q^{N_0}_{F}$, then $\xi\in S_{F}(\Upsilon_n(f))$. Let $q=\Upsilon_n(f)$, then $\norm{q(t)}\le N_0, t\in I$. Therefore, by Hypothesis (\textbf{F2}), $\norm{\xi(t)}\le \eta_{N_0}(t)$ a.a. $t\in I$.\\
	
	\textbf{\textit{Step-(v):}} It still needs to be demonstrated that the multimap $\Gamma_n$ is weakly weakly upper semicontinuous. We begin by proving that the multimap $\Gamma_n$ possesses a weakly closed graph. For this consider $\{f_m\}_{m\in \mathbb{N}}\subset Q^{N_0}_{F}$, $\{\mu_m\}_{m\in \mathbb{N}}\subset Q^{N_0}_{F}$ be such that $f_m\in \Gamma_n(\mu_m)$ for all $m\in \mathbb{N}$ and $f_m\rightharpoonup f$, $\mu_m\rightharpoonup \mu$ in $Q^{N_0}_{F}$. We prove $f\in \Gamma_n(\mu)$. The fact that $f_m\in \Gamma_n(\mu_m)$ we obtain
	\begin{equation}
		f_m\in S_{F}(\Upsilon_n(\mu_m)).
	\end{equation}  
	Since $\Upsilon_n$ map is continuous, it implies that $\Upsilon_n(\mu_m)\to \Upsilon_n(\mu)$ as $m\to \infty$. According to the weakly closed graph of the map $S_{F}$ we obtain $f\in S_{F}(\Upsilon_n(\mu))$. Therefore, $f\in \Gamma_n(\mu)$ and this concludes the proof.\\
	In the following, we establish that the map $\Gamma_n$ is weakly weakly upper semicontinuous. Assume $C$ is a weakly closed set in $Q_F^{N_0}$. We prove that the set $\{f\in Q_F^{N_0}: \Gamma_n(f)\cap C\neq \phi\}$ is weakly closed set in $Q_F^{N_0}$. Let $f_m\in Q_F^{N_0}$ be such that $\Gamma_n(f_m)\cap C \neq \phi$ for all $m\in \mathbb{N}$ and $f_m\rightharpoonup f$ in $L^2(I,X)$. Let $h_m\in \Gamma_n(f_m)\cap C$ for $m\in \mathbb{N}$. Then $h_m\in Q_F^{N_0}$ for all $m\in \mathbb{N}$ and hence $h_m\rightharpoonup h$ in $L^2(I,X)$. Since $\Gamma_n$ possesses a closed graph and $C$ is weakly closed leads to the conclusion that $h\in \Gamma_n(f)\cap C$. Hence, $\Gamma_n(f)\cap C\neq \phi$ and the map $\Gamma_n$ is weakly weakly upper semicontinuous.  \\

	\textbf{\textit{Step-(vi):}} Consequently, from Theorem  \ref{fixed}, there exists a fixed point of the map $\Gamma_n$ for each $n\in \mathbb{N}$. The definition of the map $\Gamma_n$ gives $f_n\in S_{F}(\Upsilon_n(f_n)), n\in \mathbb{N}$. Set $q_n=\Upsilon_n(f_n)$, then we have $f_n\in S_{F}(q_n), n\in \mathbb{N}$. The definition of the map $\Upsilon_n$ gives
	\begin{equation*}\label{422}
		q_n(t)=P_n\mathcal{S}_{\alpha}(t)(x_0+w_n)+\int_{0}^{t}(t-s)^{\alpha-1}P_n\mathcal{T}_{\alpha}(t-s)P_nf_n(s)ds
	\end{equation*} 
    \begin{equation}
        +\int_{0}^{t}(t-s)^{\alpha-1}P_n\mathcal{T}_{\alpha}(t-s)\mathbb{B}u_n(s)ds, t\in I,
    \end{equation}
	where $f_n\in S_{F}(q_n)$ and $u_n\in L^p(I,U)$ is given by
	\begin{equation}
		u_n=-\mathcal{W}^{-1}\mathcal{Z}_n(w_n,f_n),
	\end{equation}
	where 
	\begin{equation}
		\mathcal{Z}_n(w_n,f_n)=\mathcal{S}_{\alpha}(\nu)P_n(x_0+w_n)+\int_{0}^{\nu}(\nu-s)^{\alpha-1}\mathcal{T}_{\alpha}(\nu-s)P_nf_n(s)ds.
	\end{equation}
	We show that there is a subsequence $\{q_{n_k}\}_{k\in \mathbb{N}}$ of $\{q_n\}_{n\in \mathbb{N}}$ such that, $q_{n_k}$ converges weakly to $q\in C(I,X)$ where 
	\begin{align*}
	    q(t)=\mathcal{S}_{\alpha}(t)(x_0+w)+&\int_{0}^{t}(t-s)^{\alpha-1}\mathcal{T}_{\alpha}(t-s)f(s)ds\\
        &+\int_{0}^{t}(t-s)^{\alpha-1} \mathcal{T}_{\alpha}(t-s)\mathbb{B}u(s)ds, ~\text{for every}~ t\in I,
	\end{align*}

	where $w\in g(q), f\in S_F(q)$ and $q(\nu)=0$, for a control $u\in L^p(I,U)$.\\
	Next, we take $w_n\in g(q_n)$.
	From the hypothesis (\textbf{g1}), we deduce that $\{w_n\}_{n\in \mathbb{N}}\subset X$ is bounded in $X$ and since $X$ is reflexive, so there exists $w\in X$ for which $w_n$ weakly converges to $w$ in $X$. According to the reflexivity of $X$ and \eqref{2.4} we have
	\begin{equation}\label{3.14}
		P_n\mathcal{S}_{\alpha}(t)(x_0+w_n)\rightharpoonup \mathcal{S}_{\alpha}(t)(x_0+w)~\text{in}~ X.
	\end{equation} 
	According to the Hypothesis (\textbf{F2}), the sequence $\{f_n\}_{n\in \mathbb{N}}\subset L^2(I,X)$ is integrably bounded and the set $\{f_n(t)\}_{n\in \mathbb{N}}$ is weakly relatively compact in $X$ for almost all $t\in I$. From the result of Corollary \ref{BA} we deduce that $\{f_n\}_{n\in \mathbb{N}}$ is weakly relatively compact in $L^2(I,X)$. Therefore, there exists $f\in L^2(I,X)$ such that $f_n\rightharpoonup f$ in $L^2(I,X)$. In view of Lemma \eqref{Lem6.3} (can be found in Appendix section) we obtain $\mathbb{P}_nf_n\rightharpoonup f$ in $L^2(I,X)$, where $\mathbb{P}_n$ is given by \eqref{Projection}. Again by \eqref{2.4} 
	\begin{equation}\label{3.15}
		\int_{0}^{t}(t-s)^{\alpha-1}P_n\mathcal{T}_{\alpha}(t-s)P_nf_n(s)ds\rightharpoonup \int_{0}^{t}(t-s)^{\alpha-1}\mathcal{T}_{\alpha}(t-s)f(s)ds.
	\end{equation}
	Considering that $\{u_n\}_{n \in \mathbb{N}} \subset L^p(I,U)$ is bounded, the reflexivity of $L^p(I,U)$ implies that $u_n \rightharpoonup u$ in $L^p(I,U)$. Therefore, we have
	\begin{equation}\label{3.16}
		\int_{0}^{t}(t-s)^{\alpha-1}P_n\mathcal{T}_{\alpha}(t-s)\mathbb{B}u_n(s)ds\rightharpoonup \int_{0}^{t}(t-s)^{\alpha-1}\mathcal{T}_{\alpha}(t-s)\mathbb{B}u(s)ds.
	\end{equation}
	Combining equations \eqref{3.14}, \eqref{3.15} and \eqref{3.16} together we obtain from \eqref{422}
	\begin{equation*}
		q_n(t)\rightharpoonup q(t)=\mathcal{S}_{\alpha}(t)(x_0+w)+\int_{0}^{t}(t-s)^{\alpha-1}\mathcal{T}_{\alpha}(t-s)f(s)ds
	\end{equation*}
    \begin{equation}
        +\int_{0}^{t}(t-s)^{\alpha-1}\mathcal{T}_{\alpha}(t-s)\mathbb{B}u(s)ds, t\in I.
    \end{equation}
	The fact that $q_n(t)\rightharpoonup q(t)$ in $X$ for every $t\in I$ and $\norm{q_n}_{C(I,X)}\le N_0$, we conclude that $q_n\rightharpoonup q$ in $C(I,X)$.
	As $w_n \rightharpoonup w$ in $X$, and $w_n \in g(q_n)$, we can use hypothesis (\textbf{g2}) to conclude that $w \in g(q)$. Moreover, since $f_n$ converges weakly to $f$ in $L^2(I,X)$, $q_n(t)$ converges weakly to $q(t)$ in $X$, and $f_n \in S_F(q_n)$, we can apply Theorem \ref{SFT} to establish that $f \in S_F(q)$. Thus, $q$ is a mild solution to the problem \eqref{1.1}. \\
	The proof is finalized by proving that $q(\nu)=0$. For every $n\in \mathbb{N}$ we have
	\begin{align*}
		q_n(\nu)=&P_n\mathcal{S}_{\alpha}(\nu)(x_0+w_n)+\int_{0}^{\nu}(\nu-s)^{\alpha-1}P_n\mathcal{T}_{\alpha}(\nu-s)P_nf_n(s)ds-P_n\mathcal{W}\mathcal{W}^{-1}\mathcal{Z}_n(w_n,f_n)\\
		=&P_n\mathcal{S}_{\alpha}(\nu)(x_0+w_n)+\int_{0}^{\nu}(\nu-s)^{\alpha-1}P_n\mathcal{T}_{\alpha}(\nu-s)P_nf_n(s)ds-P_n\mathcal{Z}_n(w_n,f_n)\\
		=&P_n\mathcal{S}_{\alpha}(\nu)(x_0+w_n)+\int_{0}^{\nu}(\nu-s)^{\alpha-1}P_n\mathcal{T}_{\alpha}(\nu-s)P_nf_n(s)ds\\
		&-P_n\left[\mathcal{S}_{\alpha}(\nu)P_n(x_0+w_n)+\int_{0}^{\nu}(\nu-s)^{\alpha-1}\mathcal{T}_{\alpha}(\nu-s)P_nf_n(s)ds\right]\\
		=&P_n\mathcal{S}_{\alpha}(\nu)(x_0+w_n)-P_n\mathcal{S}_{\alpha}(\nu)P_n(x_0+w_n).
	\end{align*}
	Therefore, we have 
	\begin{equation*}
		q_n(\nu)\rightharpoonup 0.
	\end{equation*}
	Consequently, we arrive at $q(\nu)=0$, thus proving the result.
\end{proof}	

\begin{remark}
	Consider the linear fractional-order control system 
	\begin{equation}\label{3.22}
		\begin{cases}
			^{C}D^{\alpha}_{t}q(t)=\mathbb{A}q(t)+\mathbb{B}u(t), t\in I;\\
			q(0)=x_0
		\end{cases}
	\end{equation}
	is exactly controllable. It follows that the map $\mathcal{W}$ specified in \eqref{ControlW} is surjective. Thus, the control function expressed as
	\begin{equation}
		u=\mathcal{W}^{-1}\left[x_1-\mathcal{S}_{\alpha}(\nu)x_0\right]
	\end{equation}
	steers the system \eqref{3.22} from any initial state $x_0$ to any final state $x_1$ in $X$. 
\end{remark}
This allows us to prove the exact controllability results as well.
\begin{corollary}
	Assume that Hypotheses (\textbf{T}) hold and the linear fractional-order control system \eqref{3.22} is exactly controllable. Also, assume that the multivalued nonlinearity $F$ satisfies Hypotheses (\textbf{F1})-(\textbf{F4}); and the nonlocal map $g$ satisfies Hypotheses (\textbf{g1}) and (\textbf{g2}). Then, the system \eqref{1.1} is exactly controllable in $I$.
\end{corollary}
\begin{proof}
At time $\nu$, the objective is to reach the state $x_1\in X$. Consider the multi-operator $\Gamma_n: L^2(I, X)\multimap L^2(I;X)$ which is defined by
	\begin{equation}
		\Gamma_n(f)=S_F(\Upsilon_n(f)), f\in L^2(I,X),
	\end{equation}
	where for each $f\in L^2(I;X)$, $\Upsilon_n(f)\in C(I,X_n)$ satisfies 
	\begin{equation*}
		(\Upsilon_nf)(t)=P_n\mathcal{S}_{\alpha}(t)(x_0+w_n)+\int_{0}^{t}(t-s)^{\alpha-1}P_n\mathcal{T}_{\alpha}(t-s)P_nf(s)ds
	\end{equation*}
    \begin{equation}
        +\int_{0}^{t}(t-s)^{\alpha-1}P_n \mathcal{T}_{\alpha}(t-s)\mathbb{B}u(s)ds, t\in I,
    \end{equation}
	where $w_n\in g(B_C^N)$ be fixed, $f\in S_F(q)$ and $u\in L^p(I,U)$ is given by
	\begin{equation}
		u=\mathcal{W}^{-1}P_n\mathcal{Z}(w_n,f),
	\end{equation}
	where
	\begin{equation}
		\mathcal{Z}(w_n,f)=\left[x_1-\mathcal{S}_{\alpha}(\nu)(x_0+w_n)-\int_{0}^{\nu}(\nu-s)^{\alpha-1}\mathcal{T}_{\alpha}(\nu-s)f(s)ds\right].
	\end{equation}
	Applying a similar proof approach, the exact controllability results are obtained.
\end{proof}
\section{Application}
Let $\Omega = [0,\pi]$ and $I = [0,\nu]$, $0<\nu<\infty$. Consider the following Diffusion fractional control system
\begin{equation}\label{M1}
	\begin{cases}
		^{C}D^{\alpha}_tz(t,\tau)= -a(\tau)z(t,\tau)+\displaystyle b(t,\tau)\Xi\left(\tau,\int_{\Omega}\Theta(\tau)z(t,\tau)d\tau\right)+\mathbb{B}v(t,\tau), ~ (t,\tau)\in I\times \Omega\\
		z(0,\tau)=x_0(\tau), t\in I, \tau\in \Omega.
	\end{cases}
\end{equation}
Here, $a(\cdot)\in C([0,\pi], \mathbb{R}), ~ x_{0}\in L^{p}([0,\pi],\mathbb{R}), ~ 1<p\le 2 $ and $\alpha\in(\frac{1}{p},1)$. We define the operator $A : X \rightarrow X$ as 
\begin{equation}
    Ax(\tau)= -a(\tau)x(\tau), ~\tau\in [0,\pi], ~ x\in X.
\end{equation}
It is clear the map $A$ is a bounded linear operator in $X$. Hence, $A$ generates a uniformly continuous semigroup $\{T(t)\}_{t\geq 0}$ on $X$ given by 
\begin{equation}
    (T(t)x)(\tau)= e^{-ta(\tau)}x(\tau), ~ t\geq 0, ~ x\in X
\end{equation}
% Here, the operator $-\mathbb{A}_p$ generates a strongly continuous semigroup $\{T(t)\}_{t\ge 0}$ on $L^p(\mathbb{R})$ given by
% \begin{equation}
% 	T(0)f=f~\text{and}~(T(t)f)(\tau)=\frac{1}{(4\pi t)^{\frac{1}{2}}}\int_{\mathbb{R}}e^{-\frac{\abs{\tau-\theta}^2}{4t}} f(\theta)d\theta, t>0, \tau\in \mathbb{R},
% \end{equation}
% corresponds to the strongly elliptic polynomial $a:\mathbb{R}\to \mathbb{R}$ defined by $a(\tau)=\abs{\tau}^2$. Here,
% \begin{equation}
% 	D(\mathbb{A}_p)=\left\{f\in L^p(\mathbb{R}): \frac{d^2f}{d\tau^2}\in L^p(\mathbb{R})\right\}.
% \end{equation}

The function $b(t,\tau):I\times \Omega\to \mathbb{R}$ is such that
\begin{itemize}
	\item[(b1)] $t\mapsto b(t,\tau)$ is measurable for all $\tau \in\Omega$.
	\item[(b2)] $\tau\mapsto b(t,\tau)$ is continuous for a.e. $t\in I$.
	\item[(b3)] There exists $m>0$ such that 
	\begin{equation}
		\abs{b(t,\tau)}\le m~\text{a.e.}~t\in I, \tau\in \Omega.
	\end{equation}
\end{itemize}

Also, $\Theta:\Omega\to \mathbb{R}$ is a function such that $\Theta\in L^{p^{\prime}}(\Omega)$, $\frac{1}{p}+\frac{1}{p^{\prime}}=1$, and $\Xi(\cdot,\theta)\in L^p(\Omega)$ for every $\theta\in \mathbb{R}$ and satisfying
\begin{equation}\label{W2}
	\psi_1(\tau,\theta)\le \Xi(\tau,\theta)\le \psi_2(\tau,\theta)~\text{for every}~\tau\in \Omega, ~\text{for every}~\theta\in \mathbb{R},
\end{equation}
with $\psi_1, \psi_2:\Omega\times \mathbb{R}\to \mathbb{R}$ are functions such that for $i=1,2$,
\begin{equation}\label{W3}
	\psi_i(\cdot,\theta)\in L^p(\Omega)~\text{for every}~\theta\in \mathbb{R};
\end{equation}
moreover, there exists $\alpha\in L^p(\Omega,\mathbb{R}^+)$ such that
\begin{equation}\label{W5}
	\abs{\psi_i(\tau,\theta)}\le \alpha(\tau), \tau\in \Omega, \theta\in \mathbb{R};
\end{equation} 
further
\begin{equation}\label{W7}
	\psi_1(\tau,\theta_0)\le \liminf_{\theta\to \theta_0}\psi_1(\tau,\theta)~\text{and}~	\limsup_{\theta
		\to \theta_0}\psi_2(\tau,\theta)\le\psi_2(\tau,\theta_0) ,
\end{equation}
for every $\tau\in \Omega, \theta_0\in \mathbb{R}$.
In the above, we denote by $L^p(\Omega)$ the Banach space of real-valued Lebesgue integrable functions defined on $\mathbb{R}$. Concrete examples of functions that satisfy \eqref{W2}-\eqref{W7} can be found in \cite{malaguti2016nonsmooth}.

%From the properties assigned on $\Xi$, we see that jump discontinuities appear in the connection between the integral $\displaystyle\int_{\mathbb{R}}\Theta(\tau)z(t,\tau)d\tau$ and
%the control function $\Xi(\tau, r)$. Therefore, differential inclusion comes into the picture.

We now approach by rewriting the fractional control problem \eqref{M1} as an abstract problem driven by a fractional differential inclusion in the space $X=L^p(\Omega),\\
1<p\le 2$. To this aim, we put:
\begin{equation*}
	u(t)(\tau)=v(t,\tau) ~\text{and}~q(t)(\tau)=z(t,\tau), t\in I, \tau\in \Omega.
\end{equation*}
Clearly, $q:I\to L^p(\Omega)$.\\

We define $F:I\times L^p(\Omega)\multimap L^p(\Omega)$ as
\begin{equation*}
	F(t,\phi)(\tau)=b(t,\tau)\psi(\tau), (t,\phi)\in I\times L^p(\Omega),\psi\in \Upsilon(\phi),
\end{equation*}
where the multivalued map $\Upsilon: L^p(\Omega)\multimap L^p(\Omega)$ is defined as
\begin{align*}
	\Upsilon(\phi)=&\left\{\psi\in L^p(\Omega): 
	\psi_1\left(\tau,\int_{\Omega}\Theta(\omega)\phi(\omega)d\omega \right)\le \psi(\tau)\le \psi_2\left(\tau,\int_{\Omega}\Theta(\omega)\phi(\omega)d\omega\right)
	~\text{a.a.}~\tau\in \Omega\right\}.
\end{align*}
Considering the conditions \eqref{W2}-\eqref{W7}, we affirm that the multimap $\Upsilon$ is well defined.\\
Take $U=L^p(\Omega)$ and $\mathbb{B}: U\to X $ to be the identity operator.

Therefore, considering the nonlocal function $g=0$, the abstract reformulation of equation \eqref{M1} can be given as the following semilinear evolution fractional inclusion in the Banach space $X=L^p(\Omega)$:
\begin{equation*}
	\begin{cases}
		^{C}D^{\alpha}_tq(t)\in \mathbb{A}q(t)+F(t,q(t))+\mathbb{B}u(t), t\in I\\
		q(0)\in x_0+g(q).
	\end{cases}
\end{equation*}
Of course, the solutions to these equations give rise to solutions for \eqref{M1}.\\
We now verify that the hypotheses of Theorem \ref{MR} are satisfied.

\begin{proposition}\label{U}
	We now show that the fractional system
	\begin{equation}\label{C1}
		\begin{cases}
			^{C}D^{\alpha}_tq(t)= \mathbb{A}q(t)+f(t)+\mathbb{B}u(t), t\in I\\
			q(0)=x_0,
		\end{cases}
	\end{equation}
	is  null controllable in time $\nu$, for any $f\in L^2(I,X)$.
	
\end{proposition}
\begin{proof}
	By virtue of Proposition \eqref{MProp}, the null controllability of the system \eqref{C1} is equivalent to the existence of a $\gamma>0$ such that
	\begin{equation}
		\norm{(\nu-\cdot)^{\alpha-1}\mathbb{B}^*\mathcal{T}_{\alpha}^*(\nu-\cdot)\phi^*}_{L^{p^{\prime}}(I,U^*)}\ge \gamma\left[\norm{\mathcal{S}_{\alpha}^*(\nu)\phi^*}_{X^*}+\norm{(\nu-\cdot)^{\alpha-1}\mathcal{T}_{\alpha}^*(\nu-\cdot)\phi^*}_{L^2(I,X^*)}\right],
	\end{equation} 
    $\forall \phi^*\in X^*;$ or equivalently,
	\begin{equation}
		\norm{(\nu-\cdot)^{\alpha-1}\mathcal{T}_{\alpha}^*(\nu-\cdot)\phi^*}_{L^{p^{\prime}}(I,U^*)}\ge \gamma\left[\norm{\mathcal{S}_{\alpha}^*(\nu)\phi^*}_{X^*}+\norm{(\nu-\cdot)^{\alpha-1}\mathcal{T}_{\alpha}^*(\nu-\cdot)\phi^*}_{L^2(I,X^*)}\right], 
	\end{equation} 
    $\forall \phi^*\in X^*$.\\
	The authors in \cite{chalishajar2024null} proved that the fractional linear system \eqref{C1} with $f=0$ is null controllable in time $\nu$, if there is $\delta >0$ such that
	\begin{equation*}
		\norm{(\nu-\cdot)^{\alpha-1}\mathbb{B}^*\mathcal{T}_{\alpha}^*(\nu-\cdot)\phi^*}_{L^{p^{\prime}}(I,U^*)}\ge \delta\norm{\mathcal{S}_{\alpha}^*(\nu)\phi^*}, \forall \phi^*\in X^*
	\end{equation*}
    or
    \begin{equation*}
        	\norm{(\nu-\cdot)^{\alpha-1}\mathcal{T}_{\alpha}^*(\nu-\cdot)\phi^*}_{L^{p^{\prime}}(I,U^*)}\ge \delta\norm{\mathcal{S}_{\alpha}^*(\nu)\phi^*}, \forall \phi^*\in X^*
    \end{equation*}
	We can rewrite the above expression as
	\begin{equation*}
		\frac{1}{1+\delta}	\norm{(\nu-\cdot)^{\alpha-1}\mathcal{T}_{\alpha}^*(\nu-\cdot)\phi^*}_{L^{p^{\prime}}(I,U^*)}\ge \frac{\delta}{1+\delta}\norm{\mathcal{S}_{\alpha}^*(\nu)\phi^*}, 
	\end{equation*}
	From this, we conclude that
	\begin{equation*}
		\norm{(\nu-\cdot)^{\alpha-1}\mathcal{T}_{\alpha}^*(\nu-\cdot)\phi^*}_{L^{p^{\prime}}(I,U^*)}\ge \frac{\delta}{1+\delta}\left[\norm{\mathcal{S}_{\alpha}^*(\nu)\phi^*}+\norm{(\nu-\cdot)^{\alpha-1}\mathcal{T}_{\alpha}^*(\nu-\cdot)\phi^*}_{L^{p^{\prime}}(I,X^*)}\right].
	\end{equation*}
	Since $\norm{(\nu-\cdot)^{\alpha-1}\mathcal{T}_{\alpha}^*(\nu-\cdot)\phi^*}_{L^{p^{\prime}}(I,X^*)}\ge \norm{(\nu-\cdot)^{\alpha-1}\mathcal{T}_{\alpha}^*(\nu-\cdot)\phi^*}_{L^2(I,X^*)}$ we conclude that 
	\begin{equation*}
		\norm{(\nu-\cdot)^{\alpha-1}\mathcal{T}_{\alpha}^*(\nu-\cdot)\phi^*}_{L^{p^{\prime}}(I,U^*)}\ge \frac{\delta}{1+\delta}\left[\norm{\mathcal{S}_{\alpha}^*(\nu)\phi^*}+\norm{(\nu-\cdot)^{\alpha-1}\mathcal{T}_{\alpha}^*(\nu-\cdot)\phi^*}_{L^2(I,X^*)}\right],
	\end{equation*}
	from which the null controllability of the system \eqref{C1} follows.
\end{proof}

We now move forward towards Hypotheses (\textbf{F1})-(\textbf{F4}). Hypotheses (\textbf{F1}) and (\textbf{F2}) are direct consequences of the following Propositions. We refer the reader to the paper \cite{malaguti2016nonsmooth} for the proof.

\begin{proposition}\label{CV}
	For every $\phi\in L^p(\Omega)$, the set $\Upsilon(\phi)$ is nonempty, closed and convex in $L^p(\Omega)$.
\end{proposition}
\begin{proposition}\label{USC}
	The multimap $\Upsilon: L^p(\Omega)\multimap L^p(\Omega)$ is weakly compact; that means the set $\Upsilon(B)$ is relatively weakly compact in $L^p(\Omega)$ for every bounded set $B\subset L^p(\Omega)$.
\end{proposition}
The following propositions are based on continuity properties of the multivalued map $\Upsilon$. 

\begin{proposition}\label{F4}
	The multivalued map $\Upsilon:L^p(\Omega)\multimap L^p(\Omega)$ is weakly weakly upper semicontinuous.
\end{proposition}
\begin{proof}
	Following Proposition \eqref{USC}, the weak weak upper semicontinuity of $\Upsilon$ is equivalent to the graph of $\Upsilon$ being weakly closed. Hence, we show the following facts:
	
	Suppose $\{\psi_n\}_{n\in \mathbb{N}} \subset L^p(\Omega), \{\phi_n\}_{n\in \mathbb{N}} \subset L^p(\Omega)$ be such that $\psi_n\rightharpoonup \psi$, $\phi_n\rightharpoonup \phi$ in $L^p(\Omega)$ and $\psi_n\in \Upsilon(\phi_n)$. Then we show $\psi\in \Upsilon(\phi)$. The fact that $\psi_n\in \Upsilon(\phi_n)$ gives
	\begin{equation}\label{4.24}
		\psi_1\left(\tau,\int_{\Omega}\Theta(\omega)\phi_n(\omega)d\omega \right)\le \psi_n(\tau)\le \psi_2\left(\tau,\int_{\Omega}\Theta(\omega)\phi_n(\omega)d\omega\right), \tau\in \Omega.
	\end{equation}
	We show $\psi\in L^p(\Omega)$ with 
	\begin{equation}\label{4.25}
		\psi_1\left(\tau,\int_{\Omega}\Theta(\omega)\phi(\omega)d\omega \right)\le \psi(\tau)\le \psi_2\left(\tau,\int_{\Omega}\Theta(\omega)\phi(\omega)d\omega\right), \tau\in \Omega.
	\end{equation}
	As $\psi_n\rightharpoonup \psi$ in $L^p(\Omega)$, by Mazur's convexity theorem we have a subsequence 
	\begin{equation}
		\displaystyle \xi_m=\sum_{i=0}^{k_m}\alpha_{m,i}\psi_{m+i},~~ \alpha_{m,i}\ge 0,~~ \sum_{i=0}^{k_m}\alpha_{m,i}=1	
	\end{equation}
	satisfying $\xi_m\to \psi$ in $L^p(\Omega)$ and, upto subsequence, there is $N\subset \mathbb{R}$ with Lebesgue measure zero such that $\xi_m(t)\to \psi(t)$ for all $t\in \mathbb{R}\setminus N$. 
	From the expression \eqref{4.24} we compute
	\begin{equation}\label{4456}
		\sum_{i=0}^{k_m}\alpha_{m,i}	\psi_1\left(\tau,\int_{\Omega}\Theta(\omega)\phi_{m+i}(\omega)d\omega \right)\le \xi_m(\tau)\le 	\sum_{i=0}^{k_m}\alpha_{m,i}\psi_2\left(\tau,\int_{\Omega}\Theta(\omega)\phi_{m+i}(\omega)d\omega\right), 
	\end{equation}
	$\tau\in \Omega$. As $\Theta\in L^{p^{\prime}}(\Omega)$, by the definition of the weak convergence in $L^p(\Omega)$ we obtain
	\begin{equation}
		\int_{\Omega}\Theta(\omega)\phi_n(\omega)d\omega\to \int_{\Omega}\Theta(\omega)\phi(\omega)d\omega~\text{as}~n\to \infty.
	\end{equation}
	The lower semicontinuity of the function $\psi_1(\tau,\cdot)$ and the upper semicontinuity of the function $\psi_2(\tau,\cdot)$, passing limit in \eqref{4456}
	\begin{equation}
		\psi_1\left(\tau,\int_{\Omega}\Theta(\omega)\phi(\omega)d\omega \right)\le\psi(\tau) \le\psi_2\left(\tau,\int_{\Omega}\Theta(\omega)\phi(\omega)d\omega\right), \tau\in \Omega.
	\end{equation}
	This proves that the multimap $\Upsilon$ has a weakly closed graph. This completes the proof.

\end{proof}
We now verify Hypothesis (\textbf{F2}). 
\begin{proposition}\label{F5}
	For every $N\in \mathbb{N}$ there exists $\eta_N\in L^{\textcolor{red}{\frac{1}{\alpha_1}}}(I;\mathbb{R}^+)$ such that
	\begin{equation*}
		\norm{F(t,\phi)}_X\le \eta_N(t), ~\text{a.a.}~t\in I,		
	\end{equation*}
	$~\text{and for every}~\phi\in X, \norm{\phi}_X\le N$,
	with
	\begin{equation}\label{hh}
		\liminf_{N\to \infty}\frac{1}{N}\int_{0}^{\nu}\textcolor{red}{(\nu-s)^{\alpha-1}}\eta_N(s)ds=0.
	\end{equation}
\end{proposition}
\begin{proof}
	Let $\psi\in \Upsilon(\phi)$, with $\norm{\phi}_{L^p(\mathbb{R})}\le N$. The definition of the multimap $\Upsilon$ permits
	\begin{equation}
		\psi_1\left(\tau,\int_{\Omega}\Theta(\omega)\phi(\omega)d\omega \right)\le \psi(\tau)\le \psi_2\left(\tau,\int_{\Omega}\Theta(\omega)\phi(\omega)d\omega\right), ~\tau\in \Omega.
	\end{equation}
	Utilizing \eqref{W5} we obtain
	\begin{equation}
		\abs{\psi(\tau)}\le \abs{\psi_1\left(\tau,\int_{\Omega}\Theta(\omega)\phi(\omega)d\omega \right)}+\abs{\psi_2\left(\tau,\int_{\Omega}\Theta(\omega)\phi(\omega)d\omega \right)}\le 2\alpha(\tau).
	\end{equation}
	Consequently,
	\begin{equation}
		\abs{b(t,\tau)\psi(\tau)}\le 2m\alpha(\tau), \tau\in \Omega.
	\end{equation}
	Also,
	\begin{equation}
		\int_{\Omega}\abs{b(t,\tau)\psi(\tau)}^pd\tau\le 2^pm^p\int_{\Omega}\alpha^p(\tau)d\tau,
	\end{equation}
	consequently,
	\begin{equation}
		\norm{b(t,\cdot)\psi}_{L^p(\Omega)}\le 2m\norm{\alpha}_{L^p(\Omega)}.
	\end{equation}
	Define $\eta_N(t)=2m\norm{\alpha}_{L^p(\Omega)}$ and obviously \eqref{hh} holds.
\end{proof}
We are now ready to conclude this section. Based on the discussion in subsection 2.2., hypothesis (\textbf{T}) is verified. By Proposition \ref{U}, Hypothesis (\textbf{U})   holds. Further, taking into accout Propositions \ref{CV}-\ref{F5} we conclude that Hypotheses (\textbf{F1}), (\textbf{F2}), (\textbf{F4}) are satisfied. Hypothesis (\textbf{F3}) follows from the measurability of the mapping $t\mapsto b(t,\tau)$.
In conclusion, the system \eqref{M1} is null controllable in $[0, \nu]$.
\section{Appendix}
\begin{theorem}\label{Thm6.1}
	Define $\Xi:L^2(I,X)\to X$ by
	\begin{equation*}
		\Xi(f)=\int_{0}^{\nu}(\nu-s)^{\alpha-1}\mathcal{T}_{\alpha}(\nu-s)f(s)ds, f\in L^2(I,X).
	\end{equation*}
	If $f_n\rightharpoonup f$ in $L^2(I,X)$ then $\Xi(f_n)\rightharpoonup \Xi (f)$ in $X$.
\end{theorem}
\begin{proof}
	Let $x^*$ belongs to the dual of $X$. Consider the operator $	\Phi: L^2(I,X)\to \mathbb{R}$ by
	\begin{equation*}
		\Phi(f)=x^*\left(\int_{0}^{\nu}(\nu-s)^{\alpha-1} \mathcal{T}_{\alpha}(\nu-s)f(s)ds\right), f\in L^2(I,X).
	\end{equation*}
	Then, it is clear that $\Phi$ is also linear and bounded. Hence $\Phi$ belongs to the dual of $L^2(I,X)$. Therefore, the weak convergence of $f_n\rightharpoonup f$ in $L^2(I,X)$ implies
	\begin{equation*}
		\Phi (f_n)\to \Phi(f) ~\text{in}~\mathbb{R}.
	\end{equation*} 
	Hence
	\begin{equation*}
		x^*\left(\int_{0}^{\nu} (\nu-s)^{\alpha-1}\mathcal{T}_{\alpha}(\nu-s)f_n(s)ds\right)\to x^*\left(\int_{0}^{\nu}(\nu-s)^{\alpha-1} \mathcal{T}_{\alpha}(\nu-s)f(s)ds\right),
	\end{equation*}
	which gives
	\begin{equation*}
		\int_{0}^{\nu}(\nu-s)^{\alpha-1} \mathcal{T}_{\alpha}(\nu-s)f_n(s)ds\rightharpoonup \int_{0}^{\nu}(\nu-s)^{\alpha-1} \mathcal{T}_{\alpha}(\nu-s)f(s)ds~\text{in}~X.
	\end{equation*}
\end{proof}
\begin{lemma}\cite{kumbhakar2025p}\label{Lem6.2}
	For each $n\in \mathbb{N}$ define the map $\mathbb{P}_n:L^2(I,X)\to L^2(I,X)$ by
	\begin{equation}\label{Pn}
		(\mathbb{P}_nf)(t)=P_nf(t), t\in I.
	\end{equation}
	Then we have $\mathbb{P}_nf_m\rightharpoonup \mathbb{P}_nf$ in $L^2(I,X)$ whenever $f_m\rightharpoonup f$ in $L^2(I,X)$.
\end{lemma}
% \begin{proof}
% 	The function $P_n:X\to X_n$ is linear and bounded, hence the functional $\mathbb{P}_n:L^2(I,X)\to L^2(I,X)$
% 	is also linear and bounded. Therefore, for every $f^*$ in the dual of $L^2(I,X)$ the map $\Psi=f^* \circ\mathbb{P}_n$ is in the dual of $L^2(I,X)$ too and
% 	\begin{equation*}
% 		f^*\circ(\mathbb{P}_n(f_m))=\Psi(f_m)\to \Psi(f)=f^*\circ(\mathbb{P}_n(f)), ~\text{as}~ m\to \infty.
% 	\end{equation*}
% 	This proves that 
% 	\begin{equation*}
% 		\mathbb{P}_n(f_m)\rightharpoonup \mathbb{P}_n(f) ~\text{in}~ L^2(I,X)~\text{as}~m\to \infty.
% 	\end{equation*}
% \end{proof}
\begin{lemma}\cite{kumbhakar2025p}\label{Lem6.3}
	For each $n\in \mathbb{N}$ define the map $\mathbb{P}_n:L^2(I,X)\to L^2(I,X)$ by
	\begin{equation}\label{117}
		(\mathbb{P}_nf)(t)=P_nf(t), t\in I.
	\end{equation}
	Then if $f_n\rightharpoonup f$ in $L^2(I,X)$ and there exists $\eta_{\Omega}\in L^1(I,\mathbb{R}^+)$ such that\\ $\norm{f_n(t)}\le \eta_{\Omega}(t)$ a.a. $t\in I$, we have $\mathbb{P}_nf_n\rightharpoonup f$ in $L^2(I,X)$.
\end{lemma}

In {\cite{MR3320658}} Let us consider an abstract operator $S:L^{p}(I,X)\to C(I,X)$ satisfying the following conditions
	\begin{itemize}
		\item[(S1)] There exists $c\geq0$ such that
        \begin{equation}
			\norm{Sf-Sg}_{C(I,X)}\le c\norm{f-g}_{L^{p}(I,X)}.
		\end{equation}
		\item[(S2)] For any compact $K\subset X$ and sequence $\{f_n\}_{n\in \mathbb{N}}\subset L^{p}(I,X)$ such that $\{f_n(t)\}_{n\in \mathbb{N}}\subset K$ for a.a. $t\in I$, the weak convergence $f_n\rightharpoonup f_0$ in $L^{p}(I,X)$ implies $Sf_n\to Sf_0$ in $C(I,X)$.
	\end{itemize}
    
% Let us consider an abstract operator $S:L^2(I,X)\to C(I,X)$ satisfying the following conditions
% \begin{itemize}
% 	\item[(S1)] There exists $D>0$ such that \begin{equation}
% 		\norm{Sf-Sg}_{C(I,X)}\le D\norm{f-g}_{L^2(I,X)}.
% 	\end{equation}
% 	\item[(S2)] For any compact $K\subset X$ and sequence $\{f_n\}_{n\in \mathbb{N}}\subset L^2(I,X)$ such that $\{f_n(t)\}_{n\in \mathbb{N}}\subset K$ for a.e. $t\in I$, the weak convergence $f_n\rightharpoonup f_0$ implies $Sf_n\to Sf_0$.
% \end{itemize}

% \begin{theorem}{\cite{MR3320658}}\label{wwo}
%     Let $S : L^{p}(I,X)\rightarrow C(I,X)$ be an operator satisfying the conditions (S1) and (S2) . Then for every p-time semicompact sequence $\{f_n\}_{n\in \mathbb{N}}\subset L^{p}(I,X)$ the sequence $\{Sf_n\}_{n\in \mathbb{N}}$ is relatively compact in $C(I,X)$, and moreover, if $f_{n}\rightharpoonup f_{0}$ then $Sf_{n}\rightarrow Sf_{0}$
%     \end{theorem}
    
We now provide an example of an operator satisfying properties (S1) and (S2). \\
% The proof is based on the idea of Lemma 4.2.1 \cite{kamenskii2011condensing}.
\begin{lemma}\label{GG}
	Define an operator $G: L^2(I,X)\to C(I,X)$ by
	\begin{equation}\label{G}
		(Gf)(t)=\int_{0}^{t}(t-s)^{\alpha-1}\mathcal{T}_{\alpha}(t-s)f(s)ds, f\in L^2(I,X).
	\end{equation}
	Then $G$ satisfies properties (S1) and (S2).
\end{lemma}
We finally propose a useful compactness result for semicompact sequences. 
\begin{theorem}{\cite{MR3320658}}\label{MT}
    Let $S : L^{p}(I,X)\rightarrow C(I,X)$ be an operator satisfying the conditions (S1) and (S2) . Then for every p-time semicompact sequence $\{f_n\}_{n\in \mathbb{N}}\subset L^{p}(I,X)$ the sequence $\{Sf_n\}_{n\in \mathbb{N}}$ is relatively compact in $C(I,X)$, and moreover, if $f_{n}\rightharpoonup f_{0}$ then $Sf_{n}\rightarrow Sf_{0}$
    \end{theorem}
% \begin{theorem}\label{MT}
% 	Let $G:L^2(I,X)\to C(I,X)$ be an operator satisfying (S1) and (S2). Then for every semicompact sequence $\{f_n\}_{n\in \mathbb{N}}\subset L^2(I,X)$ the sequence $\{Gf_n\}_{n\in \mathbb{N}}$ is relatively compact in $C(I,X)$ and moreover, if $f_n\rightharpoonup f_0$ in $L^2(I,X)$ then $Gf_n\to Gf_0$ in $C(I,X)$.
% \end{theorem}
\begin{corollary}\label{Colo6.6}
	Consider the operator $G:L^2(I,X)\to C(I,X)$ defined in \eqref{G}. Then for every semicompact sequence $\{f_n\}_{n\in \mathbb{N}}\subset L^2(I,X)$ the sequence $\{Gf_n\}_{n\in \mathbb{N}}$ is relatively compact in $C(I,X)$ and moreover, if $f_n\rightharpoonup f_0$ in $L^2(I,X)$ then $Gf_n\to Gf_0$ in $C(I,X)$.
	
	In particular,
	\begin{equation}
		\int_{0}^{\nu}(\nu-s)^{\alpha-1}\mathcal{T}_{\alpha}(\nu-s)f_n(s)ds\to \int_{0}^{\nu}(\nu-s)^{\alpha-1}\mathcal{T}_{\alpha}(\nu-s)f_0(s)ds,~\text{in}~X,
	\end{equation}
	for every semicompact sequence $\{f_n\}_{n\in \mathbb{N}}\subset L^2(I,X)$.
\end{corollary}

% The title of section 2:

% The text in definitions and remarks should not be italicized.
% The title of your third subsection in section 2:

% Make sure all figures are clear, legible, and look good when printed in black and white.
% Figures can be in .eps, .png, .jpg, or .pdf format

%%%%%%%%%%%%%%%%%%%%%%%%%%%%%%%%%%%%%%%%%%%%%%%%%%%%%%
%          AI TOOLS, USE AND LOCATION
%%%%%%%%%%%%%%%%%%%%%%%%%%%%%%%%%%%%%%%%%%%%%%%%%%%%%%
%We follow COPE's guidelines and policies regarding the use of Artificial Intelligence (AI) tools. COPE Policy on AI tools can be found at https://publicationethics.org/cope-position-statements/ai-author.

%Authors using AI tools in the writing of a manuscript, production of images or graphical elements of the paper, or in the collection and analysis of data, must be transparent in disclosing in this section how the AI tool was used and which tool was used. Authors are fully responsible for the content of their manuscript, even those parts produced by an AI tool, and are thus liable for any breach of publication ethics. - COPE

%Disclosure instructions

%If there is nothing to disclose, there is no need to add a declaration, otherwise please declare.

%\section*{Use of AI tools declaration}
%The author(s) declare(s) they have used Artificial Intelligence (AI) tools in the creation of this article.
%AI tools used:
%How were the AI tools used? 
%Where in the article is the information located?

%The acknowledgments section should not be numbered.
\section*{Acknowledgments}
The first author also would like to thank the Council of Scientific and Industrial Research, New Delhi, Government of India \\(File No. 09/143(0954)/2019-EMR-I), for financial support to carry out his research work. The second author wishes to thank SERB DST, New Delhi, for providing a research project with project grant no- CRG/2021/003343/MS.

%%%%%%%%%%%%%%%%%%%%%%%%%%%%%%%%%%%%%%%%%%%%%%%%%%%%%%
%          7. REFERENCES SECTION
%%%%%%%%%%%%%%%%%%%%%%%%%%%%%%%%%%%%%%%%%%%%%%%%%%%%%%

%       READ THIS SECTION CAREFULLY

% Each of the references below MUST be cited in your article above. Do not include references that are not cited in your article.

% Follow the examples below carefully. We strongly suggest that you copy and paste your reference information directly into our examples.

% List all references in alphabetical order according to the first author's last name.

% Verify each URL works correctly and can be accessed properly. Your URL links should be to reputable websites. The command line for a website link begins with: \url{ }

% Do not add MR or DOI numbers to your references. AIMS production staff will add this information.

% Using BibTex is not recommended but can be handled.

\medskip
% The information below will be filled in by AIMS production staff.
Received xxxx 20xx; revised xxxx 20xx; early access xxxx 20xx.
\medskip


\begin{thebibliography}{99}

\bibitem{hilfer2000applications}
\newblock Hilfer, Rudolf,
\newblock Applications of fractional calculus in physics,
\newblock World scientific, (2000).
%%%%%%%%%%%%%%%%%%%%%%%%%%%%%%%%%%%%%%%%%

\bibitem{kilbas2006theory}
\newblock Kilbas, Anatoli{\u\i} Aleksandrovich and Srivastava, Hari M and Trujillo, Juan J,
\newblock Theory and applications of fractional differential equations,
\newblock elsevier \textbf{204} (2006), 
%%%%%%%%%%%%%%%%%%%%%%%%%%%%%%%%%%%%%%%%%
\bibitem{wang2011existence}
\newblock Wang, JinRong and Zhou, Yong,
\newblock Existence and controllability results for fractional semilinear differential inclusions,
\newblock \emph{Nonlinear Analysis: Real World Applications},
\newblock \textbf{6} Elsevier (2011), 3642--3653
%%%%%%%%%%%%%%%%%%%%%%%%%%%%%%%%%%%%%%%%%
\bibitem{yan2011nonlocal}
\newblock Yan, Zuomao,
\newblock On a nonlocal problem for fractional integrodifferential inclusions in Banach spaces,
\newblock \emph{Annales Polonici Mathematici}
\newblock \textbf{101}, (2011), 87--103 
%%%%%%%%%%%%%%%%%%%%%%%%%%%%%%%%%%%%%%%%%

\bibitem{deimling2011multivalued}
\newblock Deimling, Klaus,
\newblock Multivalued differential equations,
\newblock Walter de Gruyter \& Co., Berlin, 1992.

%%%%%%%%%%%%%%%%%%%%%%%%%%%%%%%%%%%%%%%%%

%%%%%%%%%%%%%%%%%%%%%%%%%%%%%%%%%%%%%%%%%
\bibitem{kamenskii2011condensing}
\newblock Kamenskii, Mikhail and Obukhovskii, Valeri and Zecca, Pietro,
\newblock Condensing multivalued maps and semilinear differential
inclusions in {B}anach spaces,
\newblock De Gruyter Series in Nonlinear Analysis and Applications, \textbf{7}, Walter de Gruyter \& Co., Berlin, 2001.
%%%%%%%%%%%%%%%%%%%%%%%%%%%%%%%%%%%%%%%%%%
\bibitem{matignon1996some}
\newblock Matignon, Denis and d’Andr{\'e}a-Novel, Brigitte, 
\newblock Some results on controllability and observability of finite-dimensional fractional differential systems
\newblock Computational engineering in systems applications \textbf{2} 952--956, Citeseer, 1996.
%%%%%%%%%%%%%%%%%%%%%%%%%%%%%%%%%%%%%%%%%%
\bibitem{lu2016lack}
\newblock L{\"u}, Qi and Zuazua, Enrique
\newblock On the lack of controllability of fractional in time ODE and PDE 
\newblock Mathematics of Control, Signals, and Systems \textbf{28} , Springer, 2016.
%%%%%%%%%%%%%%%%%%%%%%%%%%%%%%%%%%%%%%%%%%

\bibitem{curtain1978infinite}
\newblock Curtain, Ruth F. and Pritchard, Anthony J.,
\newblock Infinite dimensional linear systems theory,
\newblock Lecture Notes in Control and Information Sciences, \textbf{8},Springer-Verlag, Berlin-New York, 1978.
%%%%%%%%%%%%%%%%%%%%%%%%%%%%%%%%%%%%%%%%%%
\bibitem{kumbhakar2025p}
\newblock Kumbhakar, Bholanath and Pandey, Dwijendra Narain, 
\newblock $L^p$ -null controllability of an abstract differential inclusion with a nonlocal condition
\newblock \emph{Mathematical Control and Related Fields}, Mathematical Control and Related Fields (2025),







%%%%%%%%%%%%%%%%%%%%%%%%%%%%%%%%%%%%%%%%%%
\bibitem{malaguti2016nonsmooth}
\newblock Malaguti, Luisa and Rubbioni, Paola,
\newblock Nonsmooth feedback controls of nonlocal dispersal models
conditions,
\newblock \emph{Nonlinearity}, 
\textbf{29} (2016), 823-850.
%%%%%%%%%%%%%%%%%%%%%%%%%%%%%%%%%%%%%%%%%
\bibitem{yu2022solvability}
\newblock Yu, Yang-Yang and Wang, Fu-Zhang,
\newblock Solvability for a nonlocal dispersal model governed by time
and space integrals,
\newblock \emph{Open Math.}, 
\textbf{20} (2022), 1785-1799.
%%%%%%%%%%%%%%%%%%%%%%%%%%%%%%%%%%%%%%%%%
\bibitem{benedetti2013evolution}
\newblock Benedetti, Irene and Taddei, Valentina and V\"{a}th, Martin,
\newblock Evolution problems with nonlinear nonlocal boundary
conditions,
\newblock \emph{J. Dynam. Differential Equations}, 
\textbf{25} (2013), 477-503.
%%%%%%%%%%%%%%%%%%%%%%%%%%%%%%%%%%%%%%%%%
\bibitem{deng1993exponential}
\newblock Deng, Keng,
\newblock Exponential decay of solutions of semilinear parabolic
equations with nonlocal initial conditions,
\newblock \emph{J. Math. Anal. Appl.}, 
\textbf{179} (1993), 630-637.
%%%%%%%%%%%%%%%%%%%%%%%%%%%%%%%%%%%%%%%%%
\bibitem{MR3461697}
\newblock Zhou, Yong, V. Vijayakumar, and R. Murugesu,
\newblock Controllability for fractional evolution inclusions without compactness, 
\newblock \emph{Evolution Equations \& Control Theory} \textbf{4}, no. 4 (2015)
%%%%%%%%%%%%%%%%%%%%%%%%%%%%%%%%%%%%%%%%%
\bibitem{benedetti2014controllability}
\newblock Benedetti, Irene and Obukhovskii, Valeri and Taddei, Valentina, 
\newblock  Controllability for systems governed by semilinear evolution inclusions without compactness,
\newblock \emph{Nonlinear Differential Equations and Applications NoDEA} Springer \textbf{21}, (2014), 795--812
%%%%%%%%%%%%%%%%%%%%%%%%%%%%%%%%%%%%%%%%%
\bibitem{MR3320658}
\newblock Zhenhai Liu and  Biao Zeng
\newblock  Existence and controllability for fractional evolution inclusions of {C}larke's subdifferential type.
\newblock Appl. Math. Comput.
\textbf{257} 2015. 178--189.
%%%%%%%%%%%%%%%%%%%%%%%%%%%%%%%%%%%%%%%%%
\bibitem{borah2024null}
\newblock Borah, Jayanta and Hazarika, Dibyajyoti and Singh, Bhupendra Kumar,
\newblock Null controllability of fractional dynamical system with nonlocal initial condition,
\newblock \emph{Journal of Fractional Calculus and Applications}, Alexandria University, Faculty of Science, \textbf{15}, No. 2, (2024) 1--11. 
%%%%%%%%%%%%%%%%%%%%%%%%%%%%%%%%%%%%%%%%%
\bibitem{triggiani1977note}
\newblock Triggiani, Roberto,
\newblock A note on the lack of exact controllability for mild solutions
in {B}anach spaces,
\newblock \emph{SIAM J. Control Optim.}, 
\textbf{15} (1977), 407-411.
%%%%%%%%%%%%%%%%%%%%%%%%%%%%%%%%%%%%%%%%%
\bibitem{MR708645}
\newblock Louis, J.-C. and Wexler, D.,
\newblock On exact controllability in {H}ilbert spaces,
\newblock \emph{J. Differential Equations}, 
\textbf{49} (1983), 258-269.
%%%%%%%%%%%%%%%%%%%%%%%%%%%%%%%%%%%%%%%%%
\bibitem{MR2098245}
\newblock Dauer, J. P. and Mahmudov, N. I.,
\newblock Exact null controllability of semilinear integrodifferential
systems in {H}ilbert spaces,
\newblock \emph{J. Math. Anal. Appl.}, 
\textbf{299} (2004), 322-332.
%%%%%%%%%%%%%%%%%%%%%%%%%%%%%%%%%%%%%%%%%
\bibitem{ainseba2002exact}
\newblock Ainseba, Bedr'Eddine,
\newblock Exact and approximate controllability of the age and space
population dynamics structured model,
\newblock \emph{J. Math. Anal. Appl.}, 
\textbf{275} (2002), 562-574.
%%%%%%%%%%%%%%%%%%%%%%%%%%%%%%%%%%%%%%%%%
\bibitem{MR4200059}
\newblock Simpor\'{e}, Yacouba,
\newblock Null controllability of a nonlinear population dynamics with
age structuring and spatial diffusion,
\newblock Nonlinear analysis, geometry and applications, Birkh\"{a}user/Springer, Cham
(2020), 1533-1564.
%%%%%%%%%%%%%%%%%%%%%%%%%%%%%%%%%%%%%%%%%
\bibitem{anita2013analysis}
\newblock Ani\c{t}a, Sebastian,
\newblock Analysis and control of age-dependent population dynamics,
\newblock Mathematical Modelling: Theory and Applications, Kluwer Academic Publishers, Dordrecht, \textbf{11}, 2000.
%%%%%%%%%%%%%%%%%%%%%%%%%%%%%%%%%%%%%%%%%
\bibitem{MR2859866}
\newblock Tenenbaum, G\'{e}rald and Tucsnak, Marius,
\newblock On the null-controllability of diffusion equations,
\newblock \emph{ESAIM Control Optim. Calc. Var.}, 
\textbf{17} (2011), 1088-1100.
%%%%%%%%%%%%%%%%%%%%%%%%%%%%%%%%%%%%%%%%%
\bibitem{MR3917460}
\newblock Hegoburu, Nicolas and Tucsnak, Marius,
\newblock Null controllability of the {L}otka-{M}c{K}endrick system with
spatial diffusion,
\newblock \emph{Math. Control Relat. Fields}, 
\textbf{8} (2018), 707-720.
%%%%%%%%%%%%%%%%%%%%%%%%%%%%%%%%%%%%%%%%%
\bibitem{MR3063167}
\newblock Mahmudov, Nazim I.,
\newblock Exact null controllability of semilinear evolution systems,
\newblock \emph{J. Global Optim.}, 
\textbf{56} (2013), 317-326.
%%%%%%%%%%%%%%%%%%%%%%%%%%%%%%%%%%%%%%%%%
\bibitem{malaguti2021begin}
\newblock Malaguti, Luisa and Perrotta, Stefania and Taddei, Valentina,
\newblock $L^p$-exact controllability of partial differential
equations with nonlocal terms,
\newblock \emph{Evol. Equ. Control Theory}, 
\textbf{11} (2022), 1533-1564.
%%%%%%%%%%%%%%%%%%%%%%%%%%%%%%%%%%%%%%%%%
\bibitem{malaguti2019exact}
\newblock Malaguti, Luisa and Perrotta, Stefania and Taddei, Valentina,
\newblock Exact controllability of infinite dimensional systems with
controls of minimal norm,
\newblock \emph{Topol. Methods Nonlinear Anal.}, 
\textbf{54} (2019), 1001-1021.
%%%%%%%%%%%%%%%%%%%%%%%%%%%%%%%%%%%%%%%%%
\bibitem{pinaud2020controllability}
\newblock Pinaud, Matthieu F and Henriquez, Hernan R,
\newblock Controllability of systems with a general nonlocal condition,
\newblock \emph{J. Differential Equations}, 
\textbf{269} (2020), 4609-4642.
%%%%%%%%%%%%%%%%%%%%%%%%%%%%%%%%%%


%%%%%%%%%%%%%%%%%%%%%%%%%%%%%%%%%%
\bibitem{MR2024162}
\newblock Denkowski, Z. "Mig orski, S., Papageorgiou,
\newblock An Introduction to Nonlinear Analysis: Theory,
\newblock \emph{Kluwer Academic Publishers, Boston}, 
\textbf{} (2003),
%%%%%%%%%%%%%%%%%%%%%%%%%%%%%%%%%%
\bibitem{papageorgiou1997handbook}
\newblock Hu, Shouchuan and Papageorgiou, Nikolas S,
\newblock Handbook of multivalued analysis. {V}ol. {I},
\newblock Mathematics and its Applications, \textbf{419}, Kluwer Academic Publishers, Dordrecht, 1997.
%%%%%%%%%%%%%%%%%%%%%%%%%%%%%%%%%%%%%%%%%%
\bibitem{MR755330}
\newblock Aubin, Jean-Pierre and Cellina, Arrigo,
\newblock Differential inclusions: set-valued maps and viability theory,
\newblock LSpringer Science \& Business Media, \textbf{264}, 2012.
%%%%%%%%%%%%%%%%%%%%%%%%%%%%%%%%%%%%%%%%%%
\bibitem{MR3524637}
\newblock Leszek Gasi\'{n}ski and Nikolaos S. Papageorgiou
\newblock  Exercises in analysis. {P}art 2. {N}onlinear analysis.
\newblock Springer, Cham,
\textbf{} 2016.
%%%%%%%%%%%%%%%%%%%%%%%%%%%%%%%%%%%%%%%%%%


%%%%%%%%%%%%%%%%%%%%%%%%%%%%%%%%%%%%%%%%%%
	
\bibitem{MR2123706}
\newblock Vieru, Alina,
\newblock On null controllability of linear systems in {B}anach spaces,
\newblock \emph{Systems and Control Letters}, 
\textbf{54} (2005), 331-337.
%%%%%%%%%%%%%%%%%%%%%%%%%%%%%%%%%%%%%%%%%
\bibitem{gallaun2020sufficient}
\newblock Gallaun, Dennis and Seifert, Christian and Tautenhahn, Martin,
\newblock Sufficient criteria and sharp geometric conditions for
observability in {B}anach spaces,
\newblock \emph{SIAM J. Control Optim.}, 
\textbf{58} (2020), 2639-2657.
%%%%%%%%%%%%%%%%%%%%%%%%%%%%%%%%%%%%%%%%%%%
\bibitem{migorski2009quasi}
\newblock Mig\'{o}rski, Stanislaw and Ochal, Anna,
\newblock Quasi-static hemivariational inequality via vanishing
acceleration approach,
\newblock \emph{SIAM J. Math. Anal.}, 
\textbf{41} (2009), 1415-1435.
%%%%%%%%%%%%%%%%%%%%%%%%%%%%%%%%%%%%%%%%%

\bibitem{chalishajar2024null}
\newblock Chalishajar, Dimplekumar and Ravikumar, K and Ramkumar, K and Anguraj, A, 
\newblock Null controllability of Hilfer fractional stochastic differential equations with nonlocal conditions,
\newblock  \emph{Numerical algebra, control and optimization},Numerical Algebra, Control and Optimization 
\textbf{14}, No. 2, (2024), {322--338} 
%%%%%%%%%%%%%%%%%%%%%%%%%%%%%%%%%%%%%%%%%
\bibitem{kovrijkine2001some}
\newblock Kovrijkine, Oleg,
\newblock Some results related to the {L}ogvinenko-{S}ereda theorem,
\newblock \emph{Proc. Amer. Math. Soc.}, 
\textbf{129} (2001), 3037-3047.
%%%%%%%%%%%%%%%%%%%%%%%%%%%%%%%%%%%%%%%%%
\bibitem{MR463994}
\newblock Ekeland, Ivar and Temam, Roger,
\newblock Convex analysis and variational problems,
\newblock Studies in Mathematics and its Applications, Vol. 1, North-Holland Publishing Co., Amsterdam-Oxford; American
Elsevier Publishing Co., Inc., New York, 1976.
%%%%%%%%%%%%%%%%%%%%%%%%%%%%%%%%%%%%%%%%%%

\end{thebibliography}
\end{document}